\pdfoutput=1
\documentclass{siamltex1213}
\usepackage{fullpage}
\usepackage{hyperref}

\usepackage{amsmath,amssymb,amsfonts,mathrsfs,amsthm}
\usepackage{array}
\usepackage[utf8]{inputenc}
\usepackage{listings}
\usepackage{mathtools}
\usepackage{pdfpages}
\usepackage[textsize=footnotesize,color=green]{todonotes}
\usepackage{bm}
\usepackage{tikz}
\usepackage[normalem]{ulem}
\usepackage{hhline}

\usepackage{algorithm}
\usepackage[noend]{algpseudocode}
\usepackage{algorithmicx}

\usepackage{graphicx}
\usepackage{subfigure}

\usepackage{color}
\usepackage{pdflscape}
\usepackage{pifont}

\newcommand{\vect}[1]{\ensuremath\boldsymbol{#1}}

\newcommand{\td}[2]{\frac{d#1}{d#2}}
\newcommand{\pd}[2]{\dfrac{\partial#1}{\partial#2}}

\newcommand{\nor}[1]{\left\| #1 \right\|}

\newcommand{\LRp}[1]{\left( #1 \right)}
\newcommand{\LRs}[1]{\left[ #1 \right]}

\newcommand{\LRb}[1]{\left| #1 \right|}
\newcommand{\LRc}[1]{\left\{ #1 \right\}}

\newcommand{\Grad} {\ensuremath{\nabla}}
\newcommand{\Div} {\ensuremath{\nabla\cdot}}

\newcommand{\jump}[1] {\ensuremath{\LRs{\![#1]\!}}}
\newcommand{\avg}[1] {\ensuremath{\LRc{\![#1]\!}}}

\newtheorem{lemma}[theorem]{Lemma}
\newtheorem{corollary}[theorem]{Corollary}

\newcommand{\eval}[2][\right]{\relax
  \ifx#1\right\relax \left.\fi#2#1\rvert}

\newcommand{\refpyr}{\widehat{\mathcal{P}}}

\newcommand{\pyr}{{\mathcal{P}}}

\newcolumntype{C}[1]{>{\centering\let\newline\\\arraybackslash\hspace{0pt}}m{#1}}

\newcommand*\diff[1]{\mathop{}\!{\mathrm{d}#1}}

\makeatletter
\renewcommand\d[1]{\mspace{6mu}\mathrm{d}#1\@ifnextchar\d{\mspace{-3mu}}{}}
\makeatother

\author{Jesse Chan\thanks{Department of Computational and Applied Mathematics, Rice University, 6100 Main St, Houston, TX, 77005} \and T. Warburton\footnotemark[1]}
\title{Orthogonal bases for vertex-mapped pyramids}
\begin{document}

%\begin{frontmatter}
%\title{Orthogonal bases for non-affine pyramidal finite elements}
%\author{Jesse Chan\fnref{corresp_author}}
%\ead{Jesse.Chan@caam.rice.edu}
%\author{T. Warburton}
%\ead{timwar@caam.rice.edu}
%\cortext[corresp_author]{Principal Corresponding author}
%\address{Department of Computational and Applied Mathematics, Rice University, 6100 Main St, Houston, TX, 77005}
\maketitle
\begin{abstract}
Discontinuous Galerkin (DG) methods discretized under the method of lines must handle the inverse of a block diagonal mass matrix at each time step.  Efficient implementations of the DG method hinge upon inexpensive and low-memory techniques for the inversion of each dense mass matrix block.  We propose an efficient time-explicit DG method on meshes of pyramidal elements based on the construction of a semi-nodal high order basis, which is orthogonal for a class of transformations of the reference pyramid, despite the non-affine nature of the mapping.   We give numerical results confirming both expected convergence rates and discuss efficiency of DG methods under such a basis.  
\end{abstract}
%\end{frontmatter}              

%\tableofcontents
\section{Introduction}

%Computational methods employing high order finite element methods have risen in popularity in recent years.  This trend is motivated by the improved properties of high order finite elements over low order finite elements --- for example, for unsteady wave propagation problems, both the dispersion and dissipation error are much lower for high order schemes than for low order discretizations \cite{deville2002high}.  More generally, assuming a stable discretization and smooth solutions, the convergence of high order methods under mesh refinement is more rapid than that of low order methods, and under refinement in polynomial order, exponential convergence is possible.  

Mesh generation has not yet matured to the point where hexahedra-only meshes can be constructed for complex geometries.  Despite this limitation, hexahedral elements remain popular, offering significant benefits over triangular and tetrahedral elements in high order finite element methods.  For example, exploitation of the tensor-product structure allows for both simple constructions of basis functions and cubature rules, as well as fast, low-memory applications of high order operators.  An alternative to purely hexahedral meshes are hex-dominant meshes \cite{baudouin2014frontal,schoberl2012netgen}, which contain primarily hexahedral elements but also a small number of tetrahedral, wedge (prism) and pyramid elements, where wedge and pyramid elements are used as ``glue'' elements to facilitate connections between hexahedral and tetrahedral elements \cite{Demkowicz:2007:CHF:1564840, bergot2010higher, gassner2009polymorphic, karniadakis1999spectral}.  

Finite elements for the pyramid have been available since the early 1990s \cite{bedrosian1992shape, zaglmayr}, though a rigorous construction of high order bases for the pyramid has been a more recent development.  Nigam and Phillips constructed conforming exact sequence finite element spaces in \cite{nigam2012high, nigam2012numerical}, and Bergot, Cohen, and Durufle gave explicit orthogonal bases on the pyramid \cite{bergot2013higher, bergot2010higher}.  Both groups showed that, in addition to polynomials, the approximation spaces on the pyramid must contain rational functions in order for the trace spaces on the faces of the pyramid to remain polynomial, which is necessary for conformity of the global finite element space.  

\subsection{Techniques for efficient mass matrix inversion}

In \cite{klockner2009nodal}, it was shown that the computational structure of DG methods makes them well-suited for accelerators such as graphics processing units (GPUs).  Under time-explicit DG methods, a block diagonal mass matrix inverse is accounted for at each timestep.  A key observation for straight-edged simplicial elements is that each block of the mass matrix is a constant scaling of the mass matrix over a reference simplex.  As a result, it is possible to sidestep the inversion of the full mass matrix by using derivative and lift operators which are premultiplied by the inverse of the reference mass matrix and applying local scalings.  This sidesteps the storage and inversion of individual mass matrices over each element, which, due to the limited memory and reduced efficiency of general linear algebra routines on GPUs, is not expected to perform well on such accelerators.

Finite element methods typically define coordinate mappings from a reference to physical element using basis functions on the pyramid.  Entries of the mass matrix are then computed on the reference element using a change of variables factor.  For affine mappings of the reference simplex, this factor is constant, implying that only one mass matrix needs to be stored and inverted for all such simplices.  For trilinear mapped tensor product hexahedral elements, this factor is no longer constant, but it is possible to decompose the mapped mass matrix into the Kronecker product of 1D mass matrices such that this factor is constant in each tensor product direction.  An alternative procedure is to employ Lagrange polynomials at Gauss-Legendre-Lobatto (GLL) quadrature points and construct the lumped mass matrix through inexact numerical integration.  This yields the Spectral Element Method (SEM), which boasts a trivially invertible diagonal mass matrix whose entries are the GLL quadrature weights.  

Bedrosian introduced in \cite{bedrosian1992shape} low order vertex shape functions for the pyramid which are rational in nature.  Using such shape functions, transformations of the reference pyramid could be defined in terms of vertex positions of the physical pyramid.  We consider in this work physical pyramids which are images of the reference pyramid under such a map, and refer to these as \emph{vertex-mapped} pyramids, which are analogous to affine mappings of the simplex and trilinear mappings of the hexahedra.  

For vertex-mapped pyramids, however, we do not observe the advantages of either simplicial or hexahedral elements.  The construction of lumped mass matrices and GLL quadratures for non-hexahedral elements is nontrivial \cite{mulder2013new}, and the tensor product structure is absent for the pyramid.  An analogue to GLL points may not even exist for non-hexahedral domains --- Helenbrook showed that, on triangles, there does not exist a Lobatto-type quadrature rule which is both exact for polynomials of order $2N-1$ and has a number of points equal to the dimension of order $N$ polynomials \cite{helenbrook2009existence}.   Furthermore, in addition to the fact that a non-planar pyramid base produces a non-affine mapping, it was shown in \cite{bergot2010higher} that for non-parallelogram pyramid bases, the mapping is not only non-affine, but rational in the $r,s,t$ coordinates.  

Attempts to rectify the costs and complications of non-affine mapped elements have also been proposed previously in the context of curvilinear meshes.  Several methods have experimented with modifying the numerical method or formulation in order to sidestep difficulties in dealing with curvilinear and non-affine transformations.  For example, Krivodonova and Berger \cite{krivodonova2006high} extrapolate boundary conditions on curvilinear boundaries to the boundary of a mesh consisting of affine-mapped triangles.  However, while this technique allows for the efficient inversion of mass matrices on simplices where the determinant of the Jacobian is constant, it does not circumvent the presence of non-affine mappings for pyramids.  

Warburton proposed in \cite{warburton2010low, warburton2013low} a Low-Storage Curvilinear discontinuous Galerkin method (LSC-DG), where the local basis functions on each element are taken to be the reference element basis functions divided by the square root of the change of variables factor over that element.  As a result, the mass matrix is identical to the reference element mass matrix for all elements, independent of the local mapping.  The analysis in \cite{warburton2013low} includes a convergence analysis with sufficient conditions requiring elements to be asymptotically affine to attain design order convergence.  These conditions do not hold for general vertex-mapped pyramids, and it was observed in \cite{bergot2013higher} that the LSC-DG error stagnated under refinement of pyramidal meshes .

%It was noted in \cite{warburton2013low} that this phenomena was likely due to the rational nature of the determinant of the Jacobian. %, and similar examples were constructed in 2D where the approximation error stagnated with $h$-refinement.  

We present here an alternative low-memory DG method by constructing a basis which yields a diagonal mass matrix for arbitrary vertex-mapped pyramids, but spans the same space as the optimal pyramid spaces described in \cite{bergot2010higher} and \cite{nigam2012high}.  The resulting DG method on vertex-mapped pyramidal meshes provides both efficient inversion of the mass matrix and optimal rates of convergence for high order approximation spaces.  Numerical results confirm the accuracy and efficiency of this basis compared to LSC-DG and matrix-free alternatives, and computational experiments are performed to assess the performance of DG on GPUs.  

\section{High order finite elements on the pyramid}

We introduce the bi-unit right pyramid $\refpyr$ with coordinates $r,s,t$ such that
\[
r,s \in [-1, -t], \quad t\in [-1,1].  
\]
We also define the Duffy-type mapping from the bi-unit cube to the bi-unit right pyramid with coordinates $a,b,c \in [-1,1]$
\[
r = (1+a)\LRp{\frac{1-c}{2}}-1, \quad s = (1+b)\LRp{\frac{1-c}{2}}-1,\quad t= c.
\]
which has a change of variables factor of $\LRp{\LRp{1-c}/{2}}^2$.  The inverse transform is given by 
\[
a = \frac{2(1+r)}{1-t}-1, \quad b =  \frac{2(1+s)}{1-t}-1,\quad c= t.
\]
Quadrature rules for the pyramid may also be constructed by defining a quadrature rule on the bi-unit cube and applying the transform to the reference element.  

\begin{figure}
\centering
\label{fig:mapped_pyr}
\subfigure{\includegraphics[width=.35\textwidth]{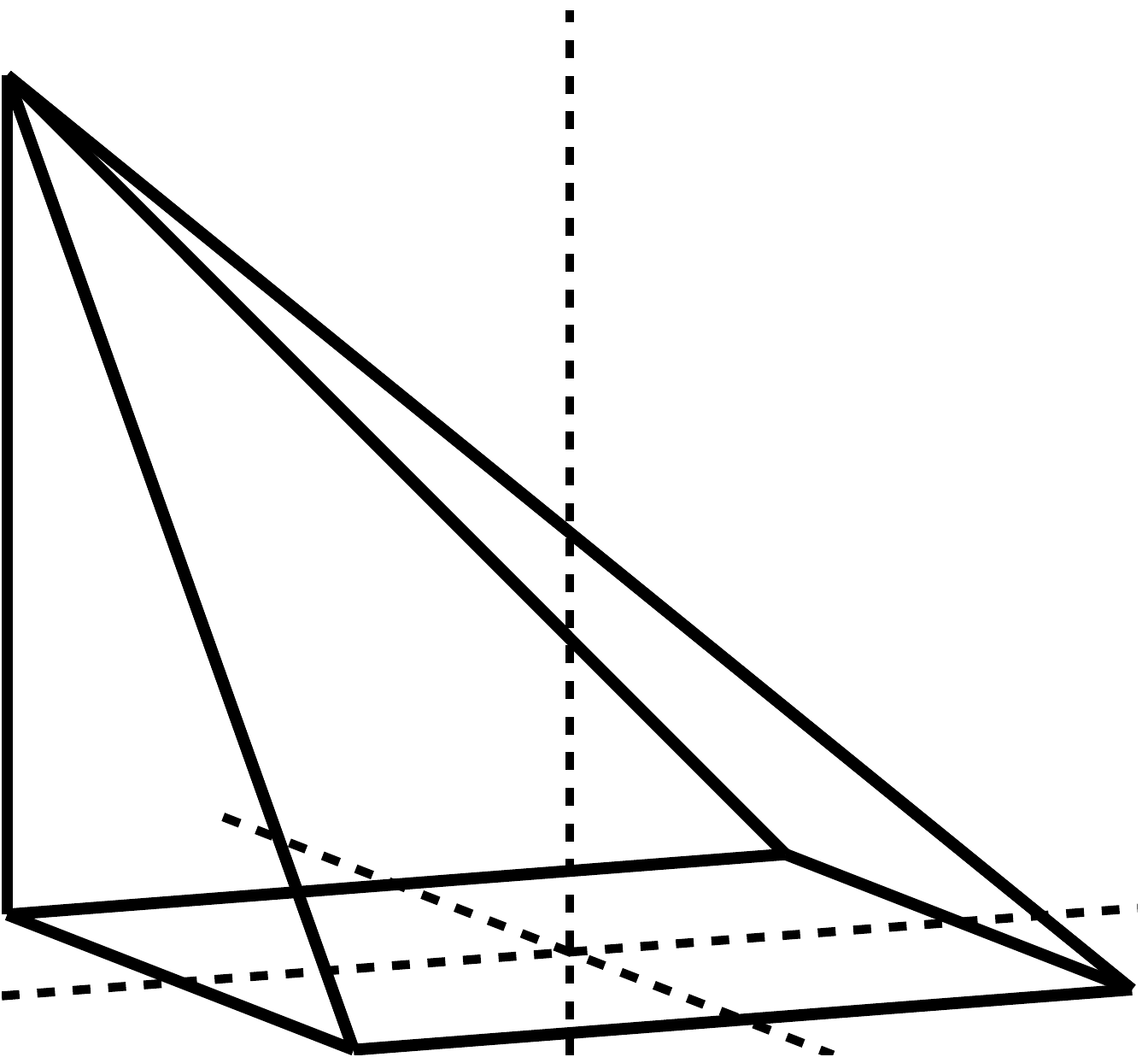}}
\hspace{2cm}
\subfigure{\includegraphics[width=.25\textwidth]{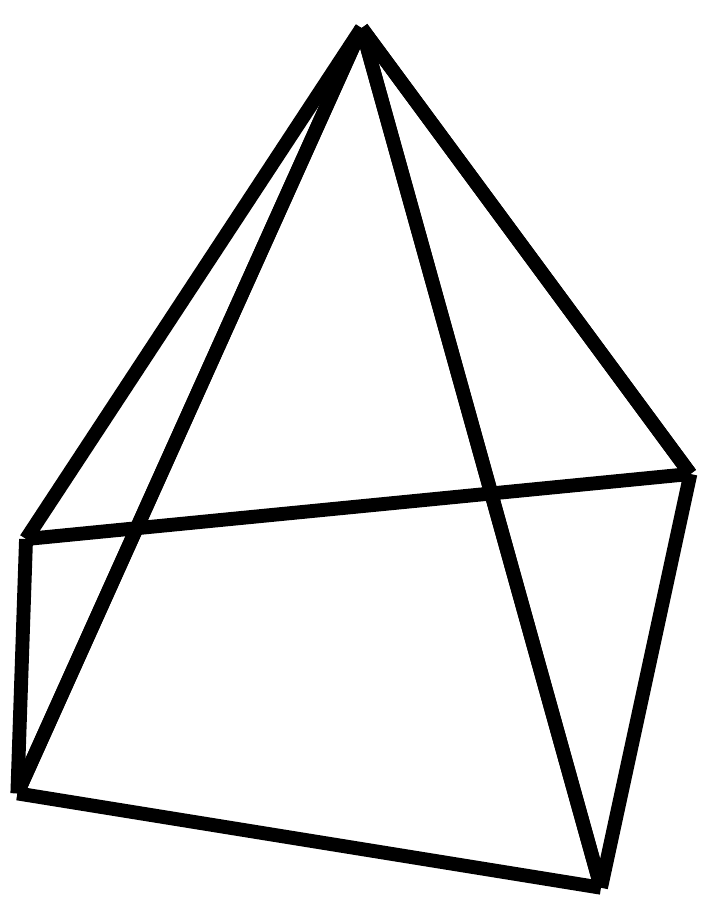}}
\caption{The reference bi-unit right pyramid (left) and an example of a vertex-mapped pyramid (right).}
\end{figure}

The vertex functions of Bedrosian \cite{bedrosian1992shape} are defined as follows on the bi-unit right pyramid:
\begin{align*}
v_1(r,s,t) &= \frac{(r+t)(s+t)}{2(1-t)}, \quad v_2(r,s,t)= -\frac{(r+t)(s+1)}{2(1-t)}, \quad v_3(r,s,t) = -\frac{(r+t)(s+t)}{2(1-t)},\\ 
v_4(r,s,t) &= \frac{(r+1)(s+1)}{2(1-t)}, \quad v_5(r,s,t) = \frac{1+t}{2}.
\end{align*}
The mapping $(x,y,z) = \vect{F}(r,s,t)$ from the reference pyramid $\refpyr$ to the physical vertex-mapped pyramid $\pyr$ is then given by 
\[
\vect{F}(r,s,t) = \sum_{i=1}^5 \vect{V}_i v_i(r,s,t),
\]
where $\vect{V}_i$ is the coordinate of the $i$th vertex of the physical pyramid.  We also define $J$, the determinant of the Jacobian of $\vect{F}$, such that 
\[
J = \LRb{\begin{array}{ccc}
\pd{\vect{F}_x}{r} & \pd{\vect{F}_x}{s} & \pd{\vect{F}_x}{t}\\ 
\vspace{-.6em}\\
\pd{\vect{F}_y}{r} & \pd{\vect{F}_y}{s} & \pd{\vect{F}_y}{t}\\
\vspace{-.6em}\\
\pd{\vect{F}_z}{r} & \pd{\vect{F}_z}{s} & \pd{\vect{F}_z}{t}
\end{array}
},\qquad \int_{\pyr} u\diff x \diff y \diff z = \int_{\refpyr} u J \diff r \diff s \diff t.
\]
%where $\pyr$ is the physical pyramid and $\refpyr$ is the reference pyramid.  

Bergot, Cohen, and Durufle \cite{bergot2010higher} defined an orthonormal basis on the bi-unit right pyramid as follows: let $P_i^{\alpha,\beta}$ be the Jacobi polynomial with weights $\alpha,\beta$.  Then, define $\psi_{ijk}$
\begin{align}
\label{eqn:rationalbasis}
\psi_{ijk}(a,b,c) &= \sqrt{2^{2\mu_{ij}+2}}P_i^{0,0}(a)P_i^{0,0}(b) \LRp{\frac{1-c}{2}}^{\mu_{ij}} P_k^{2\mu_{ij} + 2}(c),
\end{align}
where $\mu_{ij} = \max(i,j)$ and $k\leq N-\mu_{ij}$.  Under the Duffy-type mapping from $(a,b,c)$ to $(r,s,t)$ coordinates, these $\psi_{ijk}$ form an orthonormal basis over the reference bi-unit right pyramid.  

The elements of the mass matrix for the mapped pyramid $\pyr$ are defined as
\[
M_{ijk,i'j'k'} = \int_{\refpyr} \psi_{ijk} \psi_{i'j'k'} J \diff x = \int_{-1}^1\int_{-1}^1\int_{-1}^1 \psi_{ijk}\psi_{i',j',k'} \LRp{\frac{1-c}{2}}^2 J \diff a \diff b \diff c.
\]
For the orthonormal basis of Bergot, Cohen, and Durufle, the mass matrix is no longer diagonal under a non-affine mapping, and inversion of the mass matrix must be done individually over every element.  However, while it is difficult to define an orthonormal basis for an arbitrary non-affine map, it is possible to derive an orthogonal basis for a vertex-mapped transformation of the reference pyramid.  

\subsection{An orthonormal semi-nodal basis on the mapped element}

We first restate Lemma 3.5 of Bergot, Cohen, and Durufle \cite{bergot2010higher}, which gives that the determinant of the Jacobian $J$ is bilinear when mapped under the inverse Duffy-type transform to the bi-unit cube.  
%It was shown by Bergot, Cohen, and Durufle that, under the inverse mapping from the physical pyramid to the reference cube, the determinant of the Jacobian $|J|$ is in ${Q}_{1,1,0}$ for straight-edged transformations of the pyramid.  
\begin{lemma}
\label{lemma:Jab}
Let ${Q}_{N_a,N_b,N_c}$ be the space of polynomials of individual orders $N_a$, $N_b$, and $N_c$ in the $a,b,c$ coordinates on the bi-unit cube, and let $J$ be the determinant of the Jacobian mapping.  Then, $J \in Q_{1,1,0}$ for vertex-mapped transformations of the pyramid.  
\end{lemma}
We may use this fact, along with the fact that the $N+1$ point Gauss-Legendre quadrature rule integrates exactly polynomials of degree $2N+1$, to construct an orthonormal basis for the vertex-mapped pyramid (we will refer to this as the \textit{semi-nodal} basis).  To begin, we first show a property of Jacobi polynomials with varying weights.
\begin{lemma}
\label{lemma:cbasis}
For $i\neq j$,
\[
\int_{-1}^1 \LRp{\frac{1-c}{2}}^{2+(N - i)+(N-j)} P_i^{2N + 3 - 2i,0}(c) P_j^{2N + 3 - 2j,0}(c)\diff c = C^N_i\delta_{ij},
\]
where 
\[
C^N_i = \frac{N+2}{2^{2(N+1-i)}(2N+3-2i)}.
\]
\end{lemma}
\begin{proof}
Assume without loss of generality that $j < i$ and that $N, i>0$ (since $P_0$ is trivially determined for any choice of $N$).  The statement of the Lemma is then equivalent to showing
\[
\int_{-1}^1 \LRp{\frac{1-c}{2}}^{2+(N - i)+(N-j)} P_i^{2N + 3 - 2i,0}(c) p_j(c)\diff c = 0
\]
for any polynomial $p_j(c)$ of degree $j$.  Since $j < i$, we may take 
\[
p_j(c) = \LRp{\frac{1-c}{2}}^j, \qquad j = i - 1 - k
\]
for $i > k > 0$.  Then, 
\begin{align*}
&\int_{-1}^1 \LRp{\frac{1-c}{2}}^{2+2(N - i)+(N-j)} P_i^{2N + 3 - 2i,0}(c) \LRp{\frac{1-c}{2}}^j\diff c \\
&= \int_{-1}^1 \LRp{\frac{1-c}{2}}^{2N + 3- 2i} P_i^{2N + 3 - 2i,0}(c) \LRp{\frac{1-c}{2}}^k \diff c = 0,
\end{align*}
due to the weighted orthogonality of Jacobi polynomials to polynomials of lower order.  Finally, when $i=j$, we may compute
\[
C^N_i = \int_{-1}^1 \LRp{\frac{1-c}{2}}^{2 + 2N - 2i} \LRp{P^{2N+3-2i,0}_i(c)}^2\diff c = \frac{N+2}{2^{2(N+1-i)}(2N+3-2i)}.
\]
\end{proof}
A similar property was also exploited by Beuchler and Sch{\"o}berl in \cite{beuchler2006new} to construct basis functions for the triangle with sparse stiffness matrices.  
These polynomials are shown in Figure~\ref{fig:cbasis} for $N=3$.  
\begin{figure}
\centering
\label{fig:cbasis}
\includegraphics[width=.45\textwidth]{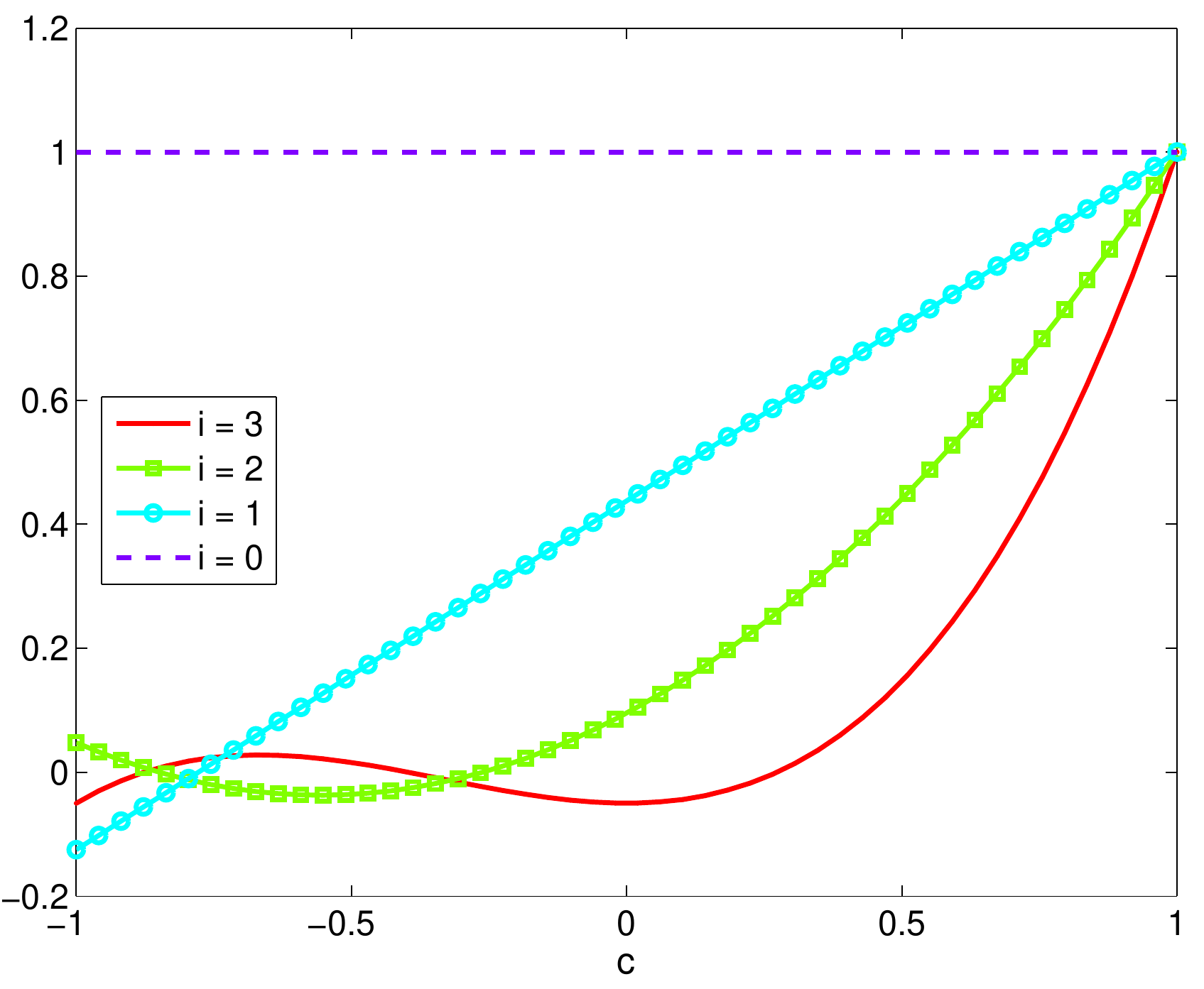}
\caption{Polynomials $P^{2N+3-2i}_i(c)$ for $N = 3$, normalized by their value at $c=1$.}
\end{figure}
Note also that a change of index from $i,j$ to $(N-i), (N-j)$ in the above proof gives 
\[
\int_{-1}^1 \LRp{\frac{1-c}{2}}^{2+ i+ j} P_{N-i}^{2i + 3,0}(c) P_{N-j}^{2j + 3,0}(c)\diff c = C^N_{N-i}\delta_{ij}, \quad C^N_{N-i} = \frac{N+2}{2^{2i+2}(2i+3)}.
\]
We may now construct a semi-nodal basis which is orthogonal on vertex-mapped pyramids by relying on the fact that the determinant of the Jacobian is bilinear in $a,b$ constant in $c$.    
\begin{lemma}
\label{lemma:orth}
Let $a_i^{k}, b^k_j$ denote $(k+1)$-point Gauss-Legendre quadrature points with corresponding weights $w_i^{k}, w_j^k$.  Let $\pyr$ be a vertex-mapped pyramid, and let $\phi_{ijk}$ be defined on the bi-unit cube as
\[
\phi_{ijk}(a,b,c) = \ell_i^{k}(a)\ell_j^{k}(b) \LRp{\frac{1-c}{2}}^{k} P^{2k+3}_{N-k}(c), %\quad 0\leq i,j \leq k, \quad 0\leq k \leq N,
\]
where $\ell_i^{k}$ is the order $k$ Lagrange polynomial which is zero at all but the $i$th $(k+1)$ Gauss-Legendre nodes, and $P^{2k+3}_{N-k}(c)$ is the Jacobi polynomial of degree $k$ with order-dependent weight $2k+3$.  Then, the $\phi_{ijk}$ are orthogonal with respect to the $L^2$ inner product over $\pyr$, and the entries of the mass matrix are
\[
M_{ijk,ijk} = {J_{ijk}w^{k}_i w^{k}_j}C^N_{N-k}
\]
where $J_{ijk}$ is the determinant of the Jacobian evaluated at quadrature points $a^{k}_i, b^{k}_j$.  
\end{lemma}
\begin{proof}
Assume without loss of generality that $k \geq k'$.  By Lemma~\ref{lemma:Jab}, the tensor product of $(k+1)$-point Gauss-Legendre quadrature rules integrates exactly 
\[
\int_{-1}^1\int_{-1}^1J \ell_i^{k}(a) \ell_{i'}^{k'}(a)\ell_j^{k}(b) \ell_{j'}^{k'}(b)\diff a\diff b.
\]
The entries of the mass matrix are then
\begin{align*}
M_{ijk,i'j'k'} &= \int_{\refpyr} \phi_{ijk}\phi_{i'j'k'} J\diff x = \int_{-1}^1\int_{-1}^1\int_{-1}^1 \phi_{ijk}\phi_{i'j'k'} \LRp{\frac{1-c}{2}}^{2}J\diff a\diff b \diff c\\
&= \sum_{l=0}^{k} \sum_{m=0}^{k} w^{k}_{l} w^{k}_{m} \ell_i^{k}(a^{k}_l)\ell_{i'}^{k'}(a_{l}^{k})\ell_j^{k}(b_m^{k})\ell_{j'}^{k'}(b_{m}) J_{lmk} \\
& \quad \times \int_{-1}^1 P^{2k+3,0}_{N-k}(c) P^{2k+3',0}_{N-k'}(c) \LRp{\frac{1-c}{2}}^{2 + k + k'}\diff c\\
&= \sum_{l=0}^{k} \sum_{m=0}^{k}J_{lmk} \delta_{il}\delta_{ii'}\delta_{jm}\delta_{jj'}  w_l^{k} w_m^{k}C^N_{N-k}{\delta_{kk'}}\\
&= J_{lmk} \delta_{ii'}\delta_{jj'}  w_l^{k} w_m^{k}C^N_{N-k}{\delta_{kk'}}.
\end{align*}
where the integral over $c$ yields $C^N_{N-k}\delta_{kk'}$ by Lemma~\ref{lemma:cbasis}.
\end{proof}

\begin{figure}
\centering
\subfigure[$k=1$]{\includegraphics[width=.325\textwidth]{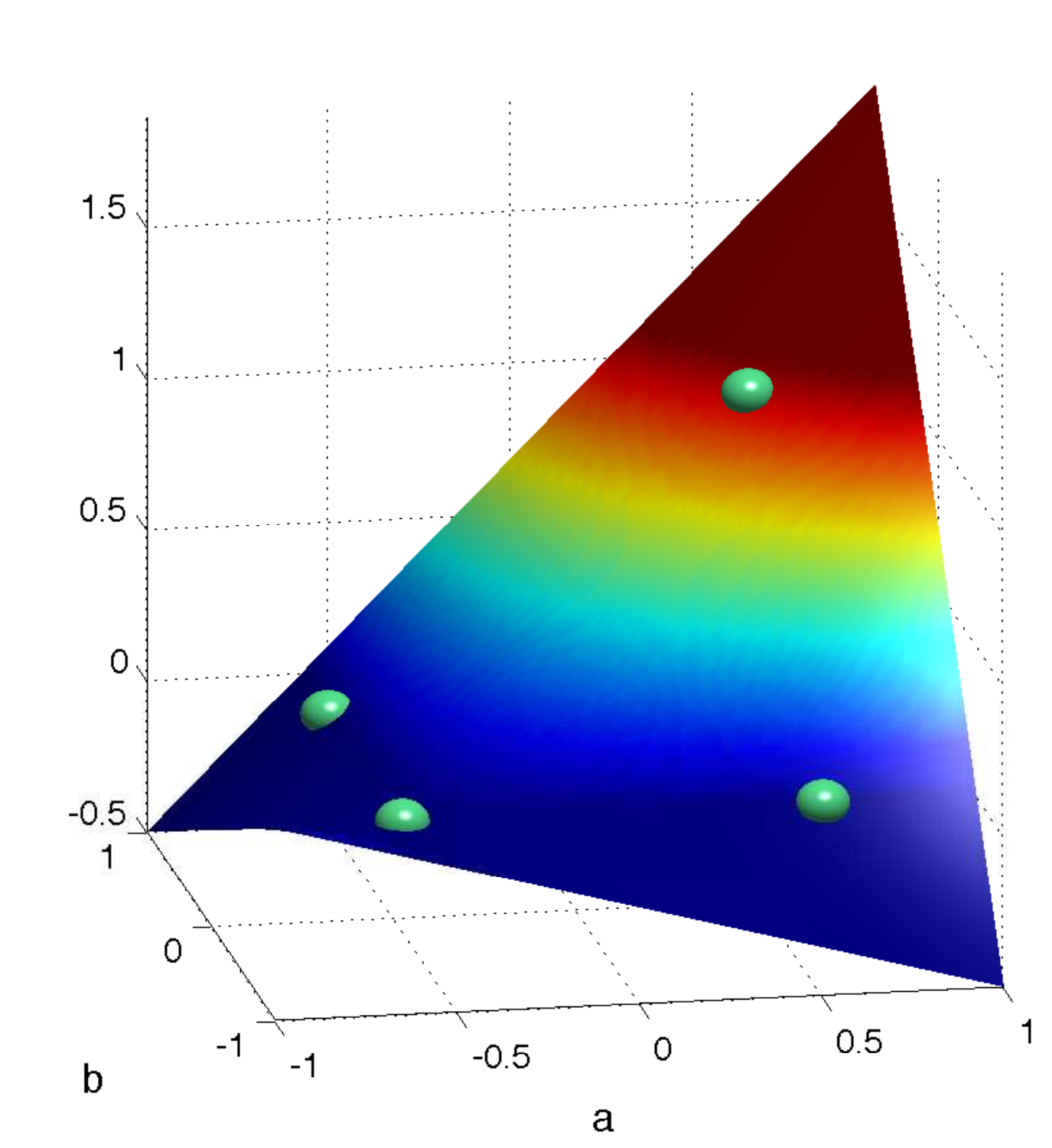}}
\subfigure[$k=2$]{\includegraphics[width=.325\textwidth]{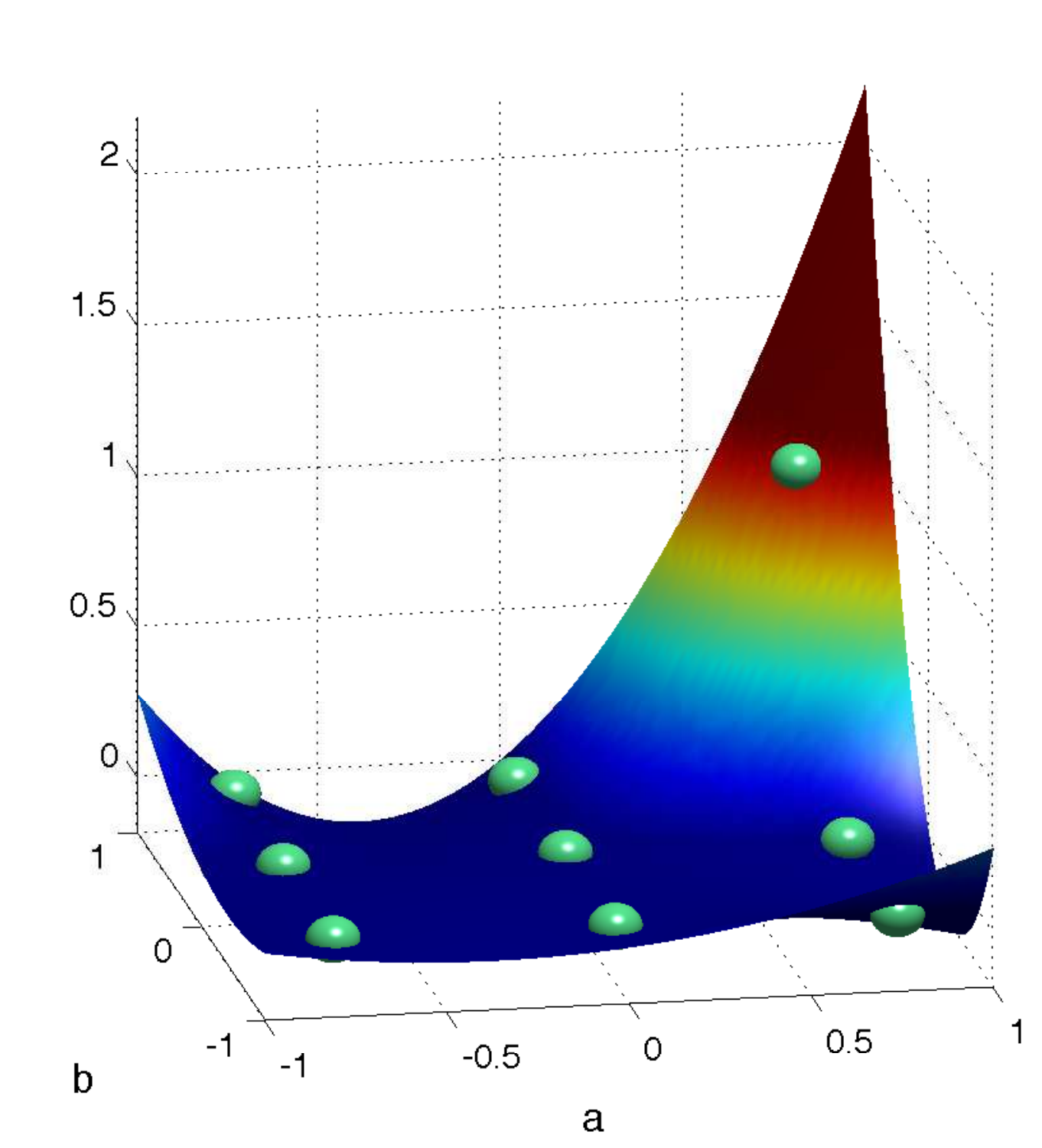}}
\subfigure[$k=3$]{\includegraphics[width=.325\textwidth]{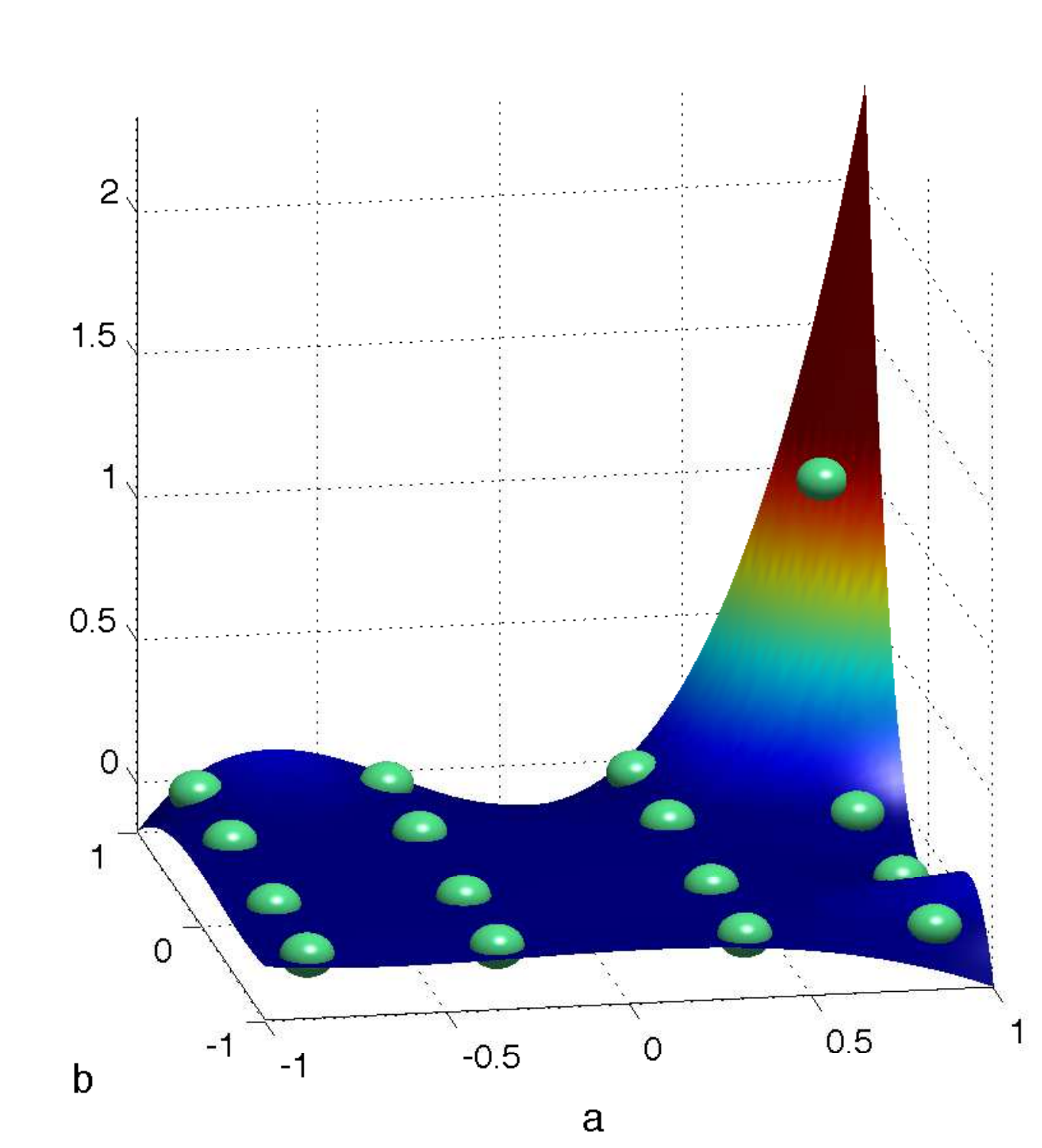}}
\caption{Polynomials $\ell^k_i(a)\ell^k_j(b)$ for $N = 3$, shown with Gauss-Legendre points overlaid.}
\label{fig:ab_basis}
\end{figure}

Figure~\ref{fig:ab_basis} shows $\ell^k_i(a)\ell_j^k(b)$, the basis in $a,b$.  The tensor product in the $c$ direction of these Lagrange polynomials with the weighted Jacobi polynomials in Figure~\ref{fig:cbasis} produces the orthogonal semi-nodal basis described in Lemma~\ref{lemma:orth}.  

\begin{lemma}
\label{lemma:span}
The semi-nodal basis defined by $\phi_{ijk}$ for $0 \leq k \leq N$ and $0\leq i,j\leq N-k$ spans the same approximation space as that of the orthonormal rational basis $\psi_{ijk}$ in Equation~(\ref{eqn:rationalbasis}).  
\end{lemma}
\begin{proof}
From Proposition 1.8 of \cite{bergot2010higher}, the orthonormal rational basis on the bi-unit cube $\psi_{ijk}$ spans the space
\[
Q_N = \sum_{k=0}^N Q_k(a,b)(1-c)^k,
\]
where $Q_k(a,b)$ consists of polynomials of order $k$ in both $a$ and $b$.  Since $\ell_i^k(a) \ell_j^k(b)$ spans $Q_k(a,b)$, and $Q_k(a,b)\supset Q_{k-1}(a,b)  \supset \ldots$, we need only to show that 
\[
{\rm span}\LRc{\LRp{\frac{1-c}{2}}^k P^{2k+3}_{N-k}(c),\quad k = 0,\ldots,N} = {\rm span}\LRc{\LRp{1-c}^k,\quad k = 0,\ldots,N } = P_N(c),
\]
where $P_N(c)$ is the space of polynomials of degree $N$ in $c$.  Since $\LRp{{1-c}/{2}}^k P^{2k+3}_{N-k}(c)$ is a polynomial of total degree $N$, it is automatically contained in $P_N(c)$, and it remains to prove the opposite inclusion.  Using a counting argument and the fact that $P_N$ has dimension $N+1$, we may prove the opposite inclusion by showing linear independence of $\LRp{{1-c}}^k P^{2k+3}_{N-k}(c)$, or equivalently $\LRp{{1-c}}^{N-k} P^{2N + 3 - 2k}_{k}(c)$ for $0\leq k\leq N$.  

Since $P^{2N + 3 - 2k}_{k}(c)$ has leading order term $c^{k}$, it is sufficient to show that 
$(1-c)^{N-k} c^{k}$ is linearly independent.  Using the binomial theorem, we may expand
\[
(1-c)^{N-k} c^{k} = \sum_{i=0}^{N-k} (-1)^i\binom {N-k}{i} c^{i+k}.
\]
The lowest order term in the above sum is $c^{k}$; since this term is distinct for each $0\leq k \leq N$, this implies linear independence of $(1-c)^{N-k} c^{k}$.
\end{proof}

The existence of an $L^2$ orthogonal basis on the vertex-mapped pyramid also allows us to characterize the spectra of the mass matrix more precisely.  
\begin{corollary}
\label{cor:mass_eigs}
The minimum and maximum eigenvalues of the mapped mass matrix under the rational basis (\ref{eqn:rationalbasis}) are given by
\[
\lambda_{\min} = J_{\min}, \qquad \lambda_{\max} = J_{\max}
\]
where $J_{\min}, J_{\max}$ are the minimum and maximum values of the determinant of the Jacobian, evaluated at the tensor product $(N+1)^2$-point Gauss-Legendre quadrature on the base of the pyramid.  
\end{corollary}
\begin{proof}
Let $M_r$ be the mass matrix constructing using the rationa basis (\ref{eqn:rationalbasis}), and let $M$ be the mass matrix constructed using the semi-nodal basis.  Scaling $\phi_{ijk}$ by $w^k_i w^k_j C^N_{N-k}$ results in an orthonormal basis, implying that $M$ is diagonal with entries equal to the values of $J$ at quadrature points $a^k_i, b^k_i$.  Lemma~\ref{lemma:span} then implies that there is a linear change of basis $S$ from the rational basis $\psi_{ijk}$ to the semi-nodal basis $\phi_{ijk}$ such that $M_r = S^{-1} M S$.  Since $M$ is diagonal, $\phi_{ijk}$ are eigenfunctions of the mass matrix with corresponding eigenvalues equal to the diagonal entries of $M$.  

Since $J \in Q_{1,1,0}$ on the bi-unit cube, $J$ is bilinear in $a,b \in [-1,1]^2$. Since a bilinear function increases or decreases monotonically in $a$ and $b$, the maximum and minimum values of $J$ are attained at the points $a^k_i, b^k_j$ closest to the boundary of $[-1,1]^2$.  Since $a^k_i, b^k_j$ are given by the $k$th order Gauss-Legendre rule for $k \leq N+1$, and the extremal points of the $k$th order Gauss-Legendre quadrature approach $-1$ and $1$ monotonically in $k$, the extremal values of $J$ are attained for $k = N+1$.  

We complete the proof by noting that the evaluation of the Jacobian factor $J$ at fixed $a^k_i,b^k_j$ is constant in $c$, so we may choose $c = -1$, which corresponds to the quadrilateral base of the pyramid.  

\end{proof}

\section{Numerical results}

We begin by comparing the semi-nodal basis constructed in Lemma~\ref{lemma:orth} with two low-storage alternatives for mass matrix inversion --- Chebyshev iteration \cite{golub1961chebyshev} and the Low-Storage Curvilinear DG method \cite{warburton2013low}.  We then demonstrate the efficiency of GPU-accelerated discontinuous Galerkin methods on vertex-mapped pyramidal elements using the new proposed basis.  

\subsection{Comparison with Chebyshev iteration}

The Chebyshev iteration constructs an explicit matrix polynomial using recurrence relations for Chebyshev polynomials, and has experienced revived attention due to the fact that it may be formulated purely in terms of matrix-vector multiplications.  This is in contrast to Conjugate Gradients, which requires inner products at each iteration.  On parallel architectures where communication between processes is costly, this necessitates a reduction from a subset of parallel processes at each iteration.  However, unlike Conjugate Gradients, the Chebyshev iteration requires a-priori knowledge of tight bounds on the spectrum of the matrix; poor estimates of the minimum and maximum eigenvalue may result in slow or stalled convergence.

The Chebyshev method has previously been used in the global inversion of continuous Galerkin mass matrices by Wathen and Rees \cite{wathen2009chebyshev}, where eigenvalue bounds for the mass matrix were derived for linear and bilinear elements in 2D.   

Suppose the spectra of the mass matrix is contained in $[\lambda_{\min},\lambda_{\max}]$.  Then, for the solution $x$ to $Mx=b$, the $k$th Chebyshev iterate $x_k$ satisfies 
\[
\nor{x - x_k}_2 \leq \LRp{\frac{2\tau^k}{1 + \tau^{2k}}}\nor{x-x_0}_2, \quad \tau = \frac{1 - \sqrt{\lambda_{\min}/\lambda_{\max}}}{1 + \sqrt{\lambda_{\min}/\lambda_{\max}}}.
\]
This above error bound may also be rearranged to yield 
\begin{equation}
\nor{x - x_k}_2 \leq 2\LRp{\frac{\sqrt{\kappa}-1}{\sqrt{\kappa}+1}}^k\nor{x-x_0}_2,
\label{eq:cheb_rate}
\end{equation}
where $\kappa$ is the mass matrix condition number.  Wathen observed that, for the symmetric, positive-definite mass matrix, the error bound for the Chebyshev iteration in (\ref{eq:cheb_rate}) is nearly identical to the error bound for the Conjugate Gradients method; the only difference between the bound for Conjugate Gradients and (\ref{eq:cheb_rate}) is the norm in which the error converges.  

\begin{figure}
\centering
\subfigure[Warped pyramid]{\includegraphics[width=.475\textwidth]{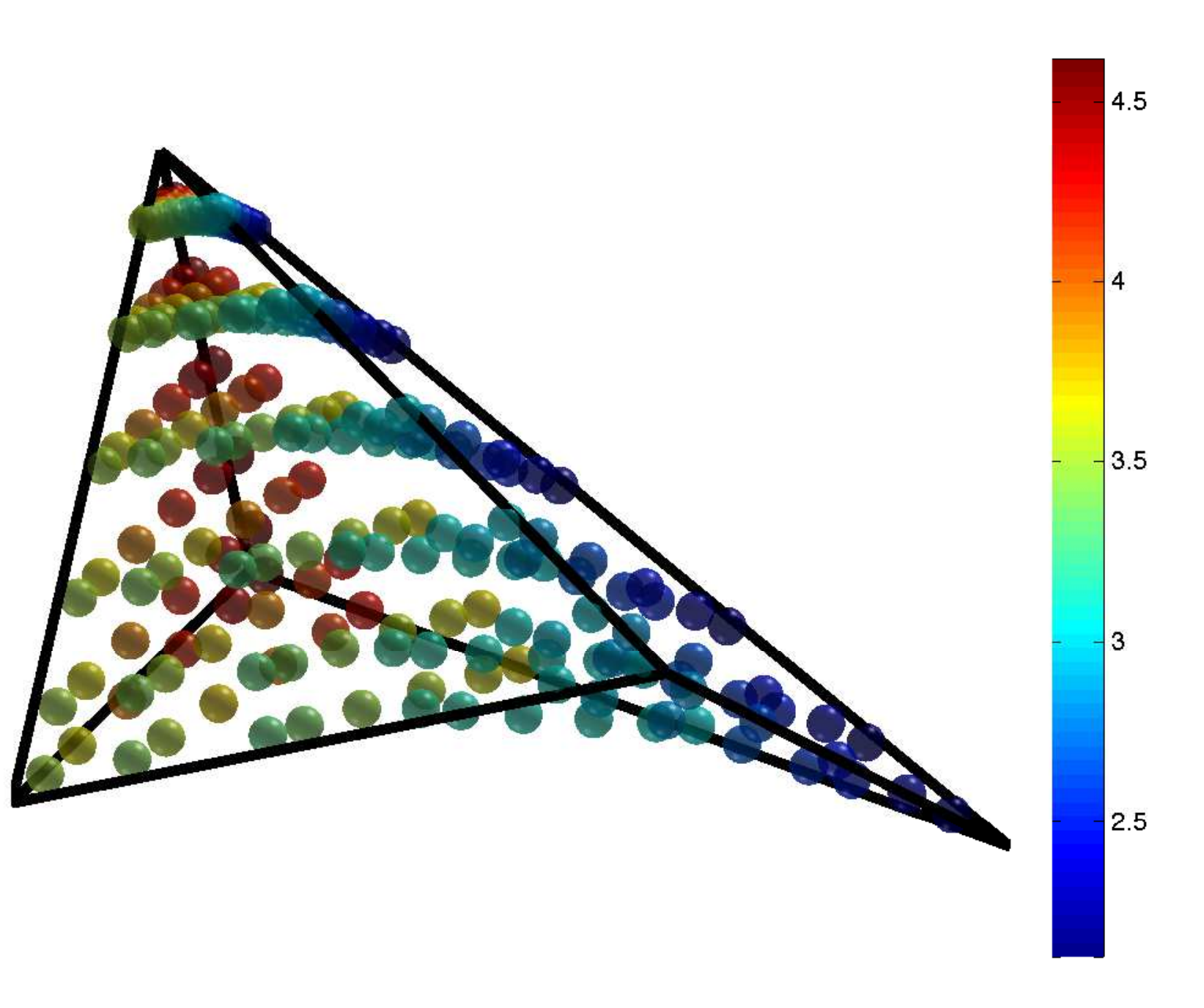}}
\subfigure[Chebyshev residual]{\includegraphics[width=.475\textwidth]{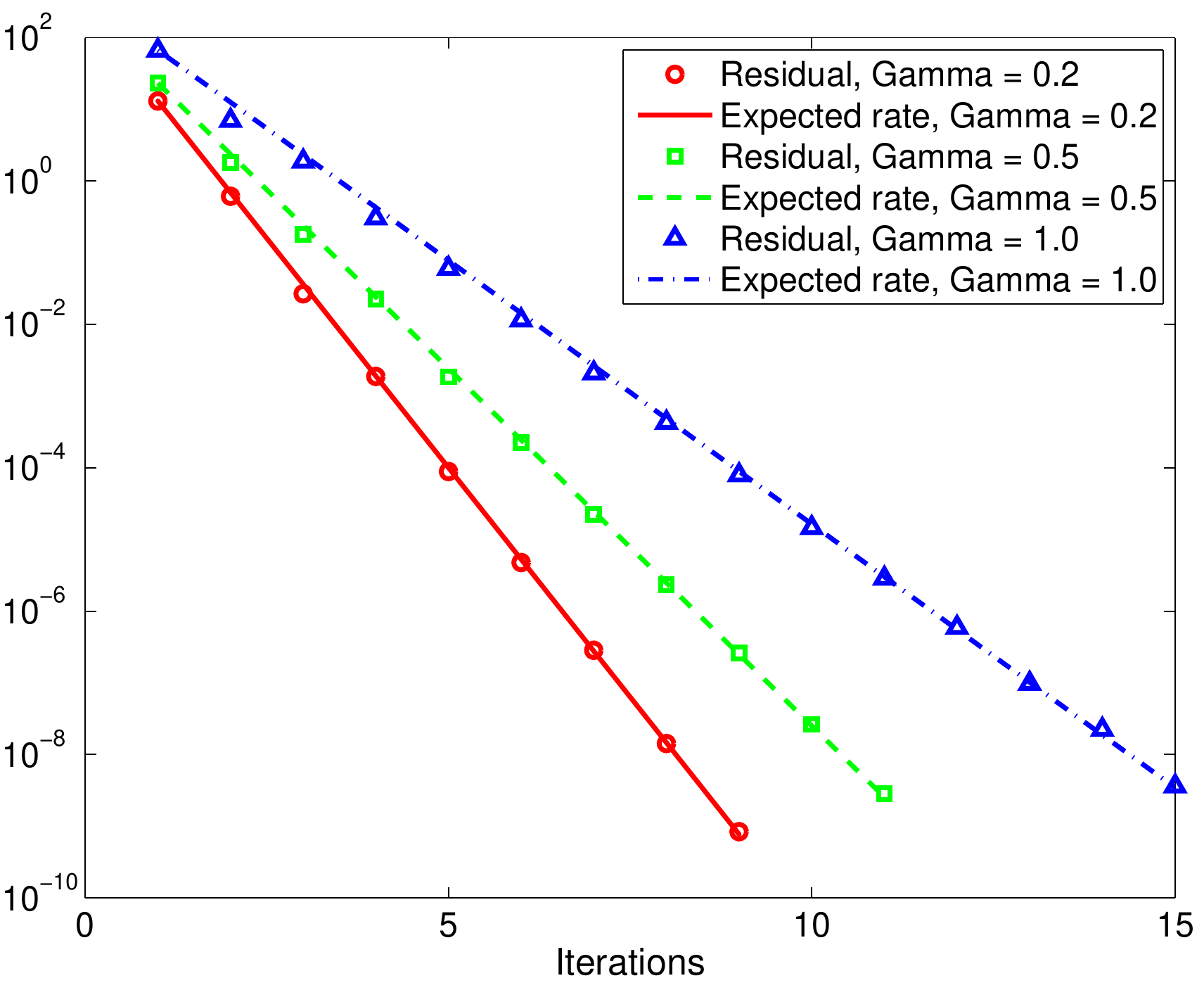}}
\caption{A warped pyramid with $\gamma = 1$, overlaid with values of the Jacobian at quadrature points for $N=5$ (left), and the convergence history of the Chebyshev iteration for various $\gamma$ (right).  The expected rate of convergence is given by (\ref{eq:cheb_rate}).}
\label{fig:warped_pyr}
\end{figure}

We consider Chebyshev iteration of the mapped mass matrix using the rational basis (\ref{eqn:rationalbasis}) of Bergot, Cohen, and Durufle.  The bounds on the maximum and minimum eigenvalues of the matrix are given by Corollary~\ref{cor:mass_eigs}, and are confirmed numerically.   We construct a mapped pyramid by warping the quadrilateral base with displacement magnitude $\gamma$, as shown in Figure~\ref{fig:warped_pyr}, along with the residual convergence of the Chebyshev iteration for various $\gamma$ and the expected rate of convergence given by (\ref{eq:cheb_rate}), which is observed to give an accurate estimate of the residual at each step.  No significant change was observed in the convergence of the Chebyshev iteration with increasing $N$; however, the results show that, even for a modestly warped pyramid, the iteration count is greater than 10, which is unacceptably high for the inversion of the mass matrix --- an $O(10)$ iteration count is relatively low for steady state or time-implicit methods, where a small number of time steps are used, but is a high cost for explicit-time discontinous Galerkin methods, which require multiple inversions of the mass matrix per time-step over millions of timesteps. 

Various preconditioners for the Chebyshev iteration were tested, with mixed results.  We observed that the diagonal of the mass matrix was constant under all vertex mappings of the pyramid, which rendered a Jacobi preconditioner ineffective.  We tested also an incomplete Cholesky factorization with tolerance $\epsilon$, which improved the number of iterations needed to reach convergence, but introduced additional memory costs for storing Cholesky factors for each element.  Additionally, the tolerance $\epsilon$ required to achieve a fixed number of iterations was observed to depend on the magnitude of the displacement $\gamma$, implying that, for fixed $\epsilon$, the effectiveness of incomplete Cholesky as a preconditioner would worsen as the shape regularity of the pyramid degrades.  For architectures such as GPUs, where device memory is typically $O(10)$ gigabytes, such additional storage costs could decrease the maximum problem size by a large factor.  

\subsection{Comparison with LSC-DG}

The Low-Storage Curvilinear DG method (LSC-DG) exploits the property of DG that local approximation spaces do not need to satisfy explicit conformity conditions.  Warburton proposed the use of specific basis functions
\[
\tilde{\phi_i}(x,y,z) = \frac{\phi_i(r,s,t)}{\sqrt{J}}, %\quad (x,y,z)\in K,
\]
where $\phi_i$ is the basis function over the reference element $\widehat{K}$, and $J$ is the determinant of the mapping Jacobian for the physical element $K$.  As a consequence, the entries of the mass matrix
\[
M_{ij} = \int_{K} \tilde{\phi_j}\tilde{\phi_i} \diff x \diff y \diff z = \int_{\widehat{K}} \frac{\phi_j\phi_i}{J} J\diff r \diff s \diff t = \int_{\widehat{K}} \phi_j\phi_i \diff r \diff s \diff t%\int_{\pyr} \tilde{\phi_j}\tilde{\phi_i} \diff x \diff y \diff z = \int_{\refpyr} \frac{\phi_j\phi_i}{J} J\diff r \diff s \diff t = \int_{\refpyr} \phi_j\phi_i \diff r \diff s \diff t
\]
are simply the entries of the mass matrix over the reference element $\widehat{K}$ \cite{warburton2010low}.  

For isoparametric curvilinear mappings, $J$ is polynomial, implying that $\tilde{\phi_i}$ is rational.  Warburton showed that, under a scaling assumption on quasi-regular elements, using such basis functions incurs an additional constant in the bounds on the best approximation error between a function $u$ and its weighted projection $\Pi_w u$. Given such an element $K$ with size $h$ and Jacobian determinant $J$, the projection error may be bounded as follows
\[
\nor{u-\Pi_w u}_{L^2(K)} \leq C h^{N+1}\nor{\frac{1}{\sqrt{J}}}_{L^{\infty}(K)}\nor{\sqrt{J}}_{W^{N+1,\infty}(K)}\nor{u}_{W^{N+1,2}(K)}
\]
where $\nor{\cdot}_{W^{N+1,\infty}(K)}$ denotes the $L^{\infty}$ Sobolev norm of order $N+1$ over $K$.  For comparison, the projection error bound for curvilinear mappings using standard mapped bases is
\[
\nor{u-\Pi_w u}_{L^2(K)} \leq Ch^{N+1} \nor{\frac{1}{\sqrt{J}}}_{L^{\infty}(K)}\nor{\sqrt{J}}_{L^{\infty}(K)}\nor{u}_{W^{N+1,2}(K)}.
\]
In other words, accuracy of approximation using rational LSC-DG basis functions comes with stricter requirements on the smoothness of the determinant of the Jacobian $J$.  

Because the mapping for pyramids with non-parallelogram bases involves factors of $(1-t)^{-1}$, the derivative of the Jacobian mapping gains higher and higher inverse powers of $(1-t)$.  Since the $W^{\infty,N+1}$ norm of $\sqrt{J}$ is ill-defined due to this singularity, the bound on LSC-DG projection error does not hold, and we are not guaranteed convergence.  
To illustrate this, we compare projections of the smooth function 
\[
f(x,y,z) = \cosh(x+y+z)
\]
on warped pyramids.  $L^2$ projections using the LSC-DG pyramid basis and the semi-nodal basis of Lemma~\ref{lemma:orth} are computed on both the warped element shown in Figure~\ref{fig:warped_pyr} for $\gamma = .2, .5, 1$, and on meshes of pyramid elements.  These meshes are constructed by subdividing the bi-unit cube into $K_{\rm 1D} \times K_{\rm 1D}\times K_{\rm 1D}$ hexahedra, where $K_{\rm 1D}$ is the number of subdivisions along each edge of the cube.  Each hexahedra is then subdivided into 6 pyramids, and the pyramid vertex positions are perturbed randomly to ensure that the determinant of the mapping Jacobian is non-constant.  $L^2$ errors are computed on each mesh at various orders of approximation $N$.
\begin{figure}
\centering
\label{fig:LSCerror}
\subfigure{\includegraphics[width=.45\textwidth]{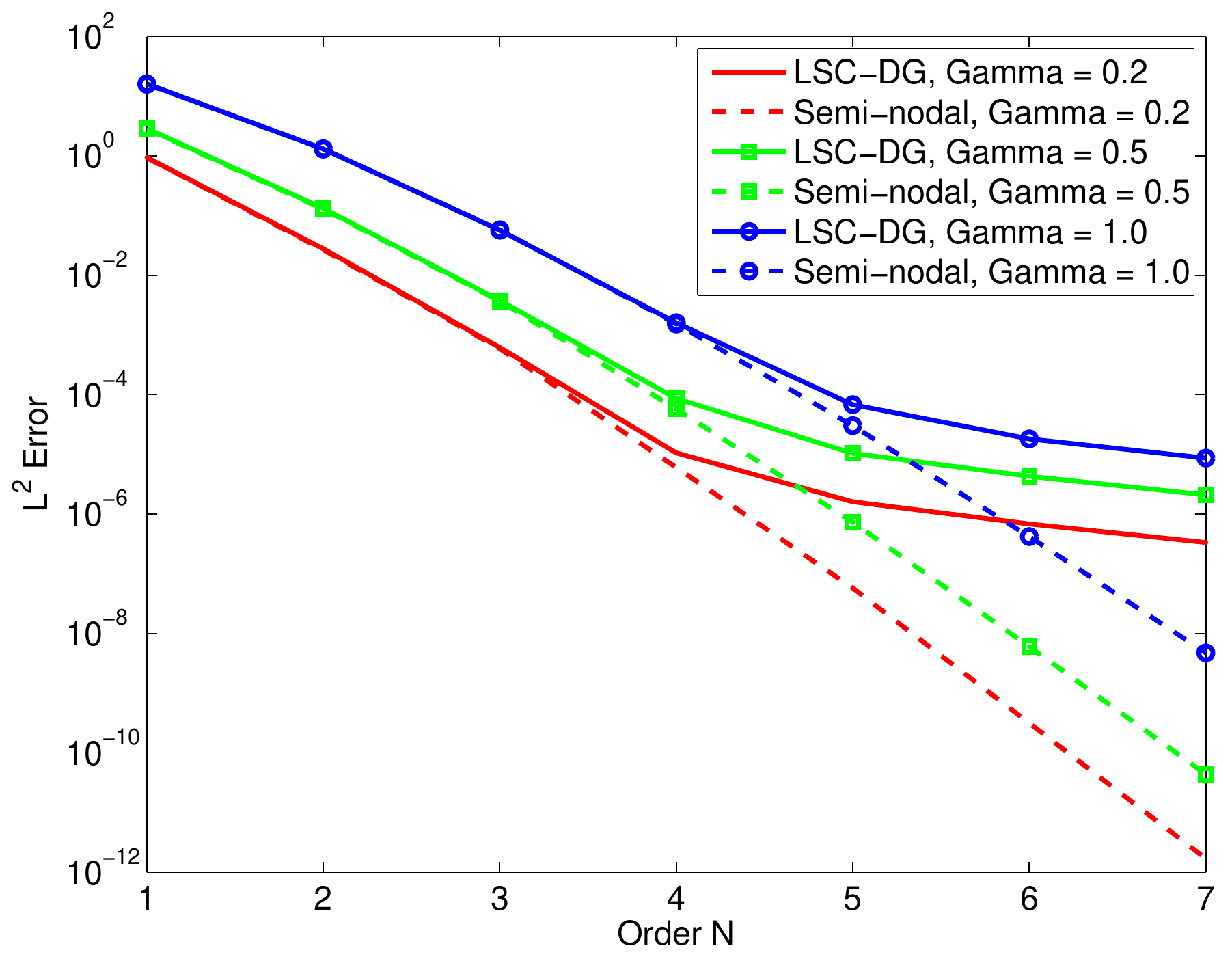}}
\subfigure{\includegraphics[width=.45\textwidth]{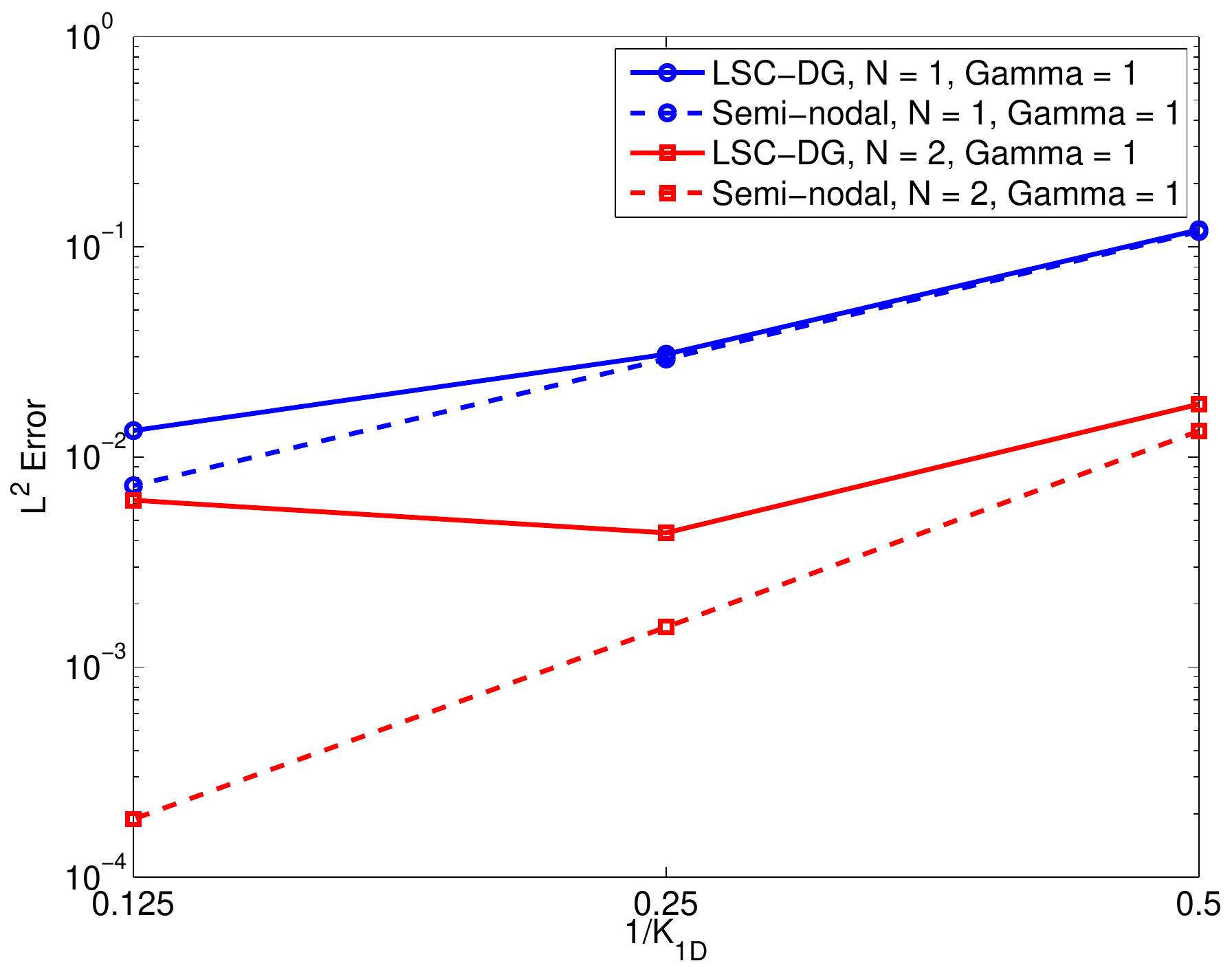}}
\caption{$L^2$ projection errors for the LSC-DG and semi-nodal orthogonal pyramid bases under increasing $N$ and various warpings of the reference pyramid (left), as well as under mesh refinement (right).  }
\end{figure}
Figure~\ref{fig:LSCerror} shows the $L^2$ error for each basis under refinement in both $h$ and $N$, with the convergence of the LSC-DG error stalling under refinement in each case.  We note that Bergot, Cohen, and Durufle also observed that error stalled under mesh refinement for fixed order $N$.  The shape regularity of the pyramid (which is controlled by $\gamma$) also affects the approximation error, but only by a constant factor.  %; though, this may also depend on the exact structure of the refinement pattern \cite{warburton2013low}.

\section{Efficient discontinuous Galerkin methods on pyramids}

Finally, to ascertain the effectiveness of the semi-nodal basis for discontinuous Galerkin methods, we examine numerical solutions of the advection equation and the acoustic wave equation using time-explicit DG.  

\subsection{Advection equation}

We consider the advection equation on a bi-unit cube $[-1,1]^3$ with periodic boundary conditions
\[
\pd{u}{t}{} + \Div \LRp{\vect{\beta} u} = 0,
\]
where $\vect{\beta}$ is a vector indicating direction of advection.  We assume a mesh $\Omega_h$ consisting purely of pyramidal elements $K$.  For each face of $K$, we refer to the outward normal as $\vect{n}$.  Let $K \in \Omega_h$ denote a specific element, and let $u^-, v^-$ denote the trace of the solution $u$ and a test function $v$, respectively, on a face.   We may then define the jump $\jump{u}$ and average $\avg{u}$ over a face as 
\[
\jump{u} = u^- - u^+, \qquad \avg{u} = \frac{u^- + u^+}{2}.
\]
Let $\vect{\beta_n} = \vect{\beta}\cdot \vect{n}$.  Then, the semidiscrete discontinuous Galerkin formulation of the advection equation is given locally as
\[
\int_K v^-\LRp{\pd{u}{t}{} + \Div \LRp{\vect{\beta}{u}}} \diff x+  \int_{\partial K} \LRp{\frac{\vect{\beta_n}-\alpha \LRb{\vect{\beta_n}}}{2}} v^- \jump{u} \diff x = 0,
\]
where $\alpha$ is a parameter.  For $\alpha = 1$, an upwind numerical flux is recovered, while for $\alpha = 0$, a central flux is recovered \cite{hesthaven2007nodal}.  The formulation is then discretized by representing $u,v$ using the semi-nodal basis $\LRc{\phi_{i}}_{i = 1}^{N_p}$ defined in Lemma~\ref{lemma:orth}.  This results in a system of ODEs
\[
\td{u}{t} +  M^{-1}\LRp{\sum_{k=1}^3 S^k u +  L^f F} = 0,
\]
where $L^f, S^k$ are the lift operator and weak derivative matrix defined by
\begin{align*}
L^f_{ij}  &= \int_{\partial K} \phi_i(x) \phi_j(x) \diff x \approx \sum_{l=1}^{N^f_c}w_l \phi_i(x_l) \phi_j(x_l) \\
S^k_{ij} &= \int_{K} \phi_i(x)\vect{\beta}_k(x) \pd{\phi_j}{x_k}  \diff x \approx  \sum_{l=1}^{N_c}w_l \phi_i(x_l) \vect{\beta}_k(x_l)\pd{\phi_j(x_l)}{x_k},
\end{align*}
and $F_l$ is the flux at the quadrature point $x_l$.  $N_c$ and $N^f_c$ denote the number of quadrature points for volume and surface integrals, respectively.  For constant advection, $N_c$ and $N^f_c$ are taken to be the minimum number of quadrature points required to integrate the mass matrix exactly.\footnote{Non-constant advection is treated identically, though it may be beneficial to increase the number of quadrature points in order to offset under-integration (aliasing) effects. }  We adopt this minimial quadrature rule, which is defined on the bi-unit cube using a $(N+1)^2$ point tensor product Gauss-Legendre quadrature in the $a,b$ coordinates and an $(N+1)$ point Gauss-Jacobi quadrature with weights $(2,0)$ in the $c$ direction.  The resulting points are then mapped using the Duffy-type transform to the bi-unit pyramid.  %We note that, due to the singularity of the determinant of the mapping Jacobian, the stiffness matrix may not be integrated exactly with standard quadrature formulas if the mapping is non-affine, though optimal (polynomial) rates of convergence are preserved when weak derivative matrices are integrated using the minimum quadrature rule necessary to integrate the mass matrix exactly \cite{bergot2010higher, nigam2012numerical}.
We note that the integrals of derivatives of rational basis functions on the pyramid may be computed exactly using the minimal quadrature rule due to cancellation of the rational Jacobian factors with change of variables factors for the derivative.  For $\alpha \in [0,1]$ and periodic boundary conditions, %the variational formulation is then skew-symmetric and the companion matrix for the system of ODEs has purely imaginary eigenvalues, and 
the DG formulation is energy stable \cite{hesthaven2007nodal}.  

The resulting system of ODEs may then be solved in time using a method of lines discretization, such as low-storage $4$th order Runge-Kutta \cite{carpenter1994fourth}. For DG on affine-mapped simplicial elements, each physical mass matrix is a constant scaling of the reference mass matrix, and $M^{-1}$ may be precomputed on the reference element and premultiplied with the lift and weak derivative matrices.  For DG on pyramids, the mass matrix differs on each element, but is diagonal under the semi-nodal pyramid basis.  Thus, instead of precomputing individual operators for each element, we precompute the diagonal factors of $M^{-1}$ and apply them at each timestep.  

\subsection{GPU acceleration} 

Typical GPU-accelerated implementations break up the solution of the system of ODEs resulting from the DG discretization into three steps: computation of volume integrals, surface integrals, and a Runge-Kutta update step, which are performed by \verb+VolumeKernel+,  \verb+SurfaceKernel+, and \verb+UpdateKernel+, respectively:
\[
\underbrace{\td{u}{t} + M^{-1}}_{\verb+UpdateKernel+}\left( \underbrace{\sum_{k=1}^3 S^k u}_{\verb+VolumeKernel+} +  \underbrace{L^f F}_{\verb+SurfaceKernel+} \right)= 0.
\]
This implementation differs slightly from simpler GPU-accelerated implementations of DG methods, in that \verb+UpdateKernel+, in addition to advancing forward in time, applies the inverse of the mass matrix and interpolates the solution to surface cubature points.  Work is partitioned such that elements (or batches of elements) are assigned to independent work-groups, while each work-item/thread processes work for either a single basis function or cubature node.  

We consider first, for simplicity of presentation, a purely modal DG method for the pure convection equation with $\beta = [1,0,0]^T$, and outline the approach used to implement a solver on the GPU.  The resulting system of equations for this specific constant advection problem is then
\[
\td{u}{t} +  M^{-1}\LRp{S^x u +  L^f F} = 0, \qquad L^f_{ij}  = \int_{\partial K} n_x\phi_i(x) \phi_j(x) \diff x, \qquad S^x_{ij} = \int_{K} \phi_i(x) \pd{\phi_j}{x}  \diff x.
\]

\subsubsection{Volume kernel}

\begin{algorithm}
\begin{algorithmic}[1]
\Procedure{Volume kernel}{}
%\State Load in $V^T$, derivative matrices $D^r, D^s, D^t$, and solution coefficients $u$.  
\State Compute derivatives at cubature points $\vect{x}_i$ for $i = 1,\ldots, N_c$.  
\[
\pd{u(\vect{x}_i)}{x}{} = \sum_{j = 1}^{N_p} \LRp{D^r_{ij} u_j \pd{r(\vect{x}_i)}{x} + D^s_{ij} u_j \pd{s(\vect{x}_i)}{x} + D^t_{ij} u_j \pd{t(\vect{x}_i)}{x}}.
\]
\State Scale by premultiplied values of $w_i J_i$, compute integral by multiplying by $V^T$.
\[
\int_{K} \phi_i \pd{u}{x}{} = \sum_{i=1}^{N_c} V^T_{ji}  w_i J_i \pd{u(\vect{x}_i)}{x}{}
\]
\EndProcedure
\end{algorithmic}
\caption{Computation of volume integrals.}
\label{alg:vol}
\end{algorithm}

The computation of volume integrals requires evaluation of solution values at cubature nodes and computation of quadrature sums.  We assume, for the reference pyramid, $N_c$ volume cubature points $r_i, s_i, t_i$ and weights $w_i$.  Let $V$ represent the volume Vandermonde matrix, and let $D^r, D^s, D^t$ represent the derivative matrices with respect to reference coordinates $r,s,t$:
\[
V_{ij} = \phi_j(r_i,s_i,t_i), \quad D^r_{ij} = \pd{\phi_j\LRp{r_i,s_i,t_i}}{r}{}, \quad D^s_{ij} = \pd{\phi_j\LRp{r_i,s_i,t_i}}{s}{}, \quad D^t_{ij} = \pd{\phi_j\LRp{r_i,s_i,t_i}}{t}{}.
\]
We store the above $N_c \times N_p$ matrices, as well as $V^T$, only for the reference element.  By storing geometric change-of-variables factors $\pd{r}{x}{}, \pd{r}{y}{},\pd{r}{z} \ldots$ and determinants $\LRb{J_i}$ of Jacobian mappings at the $N_c$ volume cubature points at each element, we may compute the integral
\[
(S^x u)_i = \int_{K} \phi_i \pd{u}{x}{}  \diff x \diff y \diff z = \int_{\widehat{K}} \phi_i(r,s,t) \LRp{\pd{u}{r}{} \pd{r}{x}{} + \pd{u}{s}{} \pd{s}{x}{} + \pd{u}{t}{} \pd{t}{x}{} } J \diff r \diff s \diff t % = \sum_{l=1}^{N_c} w_l \LRb{J_l} \phi_i(x_l) \LRp{\sum_{k=1}^3\pd{u(x)}{\widehat{x}^k}{} \pd{\widehat{x}^k}{x^k}{} }
\]
using $V, D^r, D^s D^t$ and the Jacobian factors premultiplied by quadrature weights $w_iJ_i$, as described in Algorithm~\ref{alg:vol}.  

\subsubsection{Surface Kernel}

We assume $N^f_c$ total surface cubature points (over all the faces of the pyramid) $r^f_i,s^f_i,t^f_i$ with surface cubature weights $w^f_i$, and we define $V^f$ as the surface Vandermonde matrix $V^f_{ij} = \phi_j\LRp{r^f_i,s^f_i,t^f_i}$.  Storing normals $n_{x,i}, n_{y,i}, n_{z,i}$ and determinants of Jacobian mappings $J^f_i$ at surface cubature points $\widehat{x}^f_i$, we may compute surface integrals of the flux $F$
\[
(L^f F)_i = \int_{\partial K} \phi_i F \diff x, \quad F=\LRp{n_x - \alpha\LRb{n_x}}\jump{u},
\]
as described in Algorithm~\ref{alg:surf}.  This approach differs slightly from that of standard nodal DG algorithms in that it does not loop over faces, but computes over all cubature points on the surface of a pyramid at once.  This is due to the inhomogeneous nature of the faces on a pyramid --- while it is possible to use low-memory techniques (discussed in more detail in Section~\ref{sec:improvements}) to compute surface integrals on triangular and quadrilateral faces, they require differentiation between the types of faces within a kernel, or separate kernels for triangular and quadrilateral faces.  

\begin{algorithm}
\begin{algorithmic}[1]
\Procedure{Surface kernel}{}
%\State Load in $(V^f)^T$ and solution at local face cubature points $u_f$.  
\State Compute flux $F=\LRp{n_x - \alpha\LRb{n_x}}\jump{u}$ at face cubature points, scale by surface Jacobian factors and weights $J^f, w^f$.
\[
w^f_i J^f_i F_i.
\]
\State Compute integral by multiplying by $(V^f)^T$
\[
(L^f F)_i = \int_{\partial K} \phi_i F(u_f) = \sum_{i=1}^{N_c} (V^f)^T_{ji}  w^f_i J^f_i{F}_i.
\]
\EndProcedure
\end{algorithmic}
\caption{Computation of surface integrals.}
\label{alg:surf}
\end{algorithm}

\subsubsection{Update Kernel}

Given the diagonal entries of the mass matrix, the timestep $dt$, and the $k$th step RK coefficients $r^k_a, r^k_b$, the update kernel inverts the mass matrix, performs both a Runge-Kutta substep to march the solution to the next time, and interpolates the new solution to surface cubature nodes for use in the next surface kernel, as shown in Algorithm~\ref{alg:rk}.  

\begin{algorithm}
\begin{algorithmic}[1]
\Procedure{Runge-Kutta update step}{}
%\State Load solution coefficients $u$, residual $r_i$, and inverse of mass diagonal $1/M_{ii}$.  
%\State Compute $b^k_i$ as the sum of volume/surface integrals
%\[
%b^k_i = \sum_{k=1}^3\int_K \phi_i(x) \vect{\beta}_k(x) \pd{u^k}{x_k}{} \diff x + \int_{\partial K} \phi_i F(u_f^k) \diff x.
%\]
\State Compute right-hand side 
\[
b^k_i = \frac{1}{M_{ii}}b_i,
\]
where $b_i$ is the sum of volume and surface integrals.  
\State Update residual $r$ and local solution at the $k$th RK step
\[
r_i = r^k_a r_i + dt  b^k_i, \qquad u^k_i = u_i + r^k_b r_i.  
\]
\State Interpolate local solution to face cubature points using the face Vandermonde matrix $V^f$.
\[
u^{f,k} = V^f u^k.
\]
\EndProcedure
\end{algorithmic}
\caption{Runge-Kutta update step with added interpolation to face cubature points.}
\label{alg:rk}
\end{algorithm}

 \subsubsection{Kernel optimization}

%The performance of the above discussed kernels is very sensitive to the specific hardware used, memory access, and computational parameters (such as those that govern division of work among work-groups and threads).  However, there are 

We attempted to optimize the above kernels by minimizing the number of memory accesses, minimizing non-unit strided memory accesses, maximizing the speed at which data is accessed, or hiding the effect of latency in accessing data.  
%To minimize the number of memory accesses, information is laid out in a manner such that adjacent threads access memory in a contiguous fashion, which is referred to as coalesced reads.  Since each memory access fetches a contiguous chunk of memory into cache, by maximizing the number of threads that can read from this retrieved memory, we maximize the work that may be done before the next access.  For example, dense matrices are stored in column-major format so that each thread accesses a different row in a matrix column, whose entries are contiguous in memory.  

When data must be accessed repeatedly, we take advantage of the GPU memory hierarchy.  Data that is used repeatedly within a workgroup is loaded to shared memory, and data used repeatedly in threads is loaded to register memory, both of which allow for fast data retrieval.\footnote{Since the shared memory available on GPUs is limited, using a large amount shared memory in a kernel will reduce the number of concurrent active work groups, so we do not load derivative matrices and interpolation operators to shared memory due to their large number of entries in 3D.  These matrix reads are still relatively efficient due to coalescing and caching effects.}  Likewise, since register memory is limited, multiple inputs are concatenated into strided arrays in order to decrease register pressure.  

Finally, we may hide the latency present in memory accesses by exploiting work which may be done concurrently.  For example, we may prefetch values (such as geometric factors in the computation of volume integrals) before executing other independent commands.  Additionally, prefetching may facilitate additional compiler optimizations; though the volume kernel achieves roughly the same performance with and without prefetching under CUDA, the volume kernel with prefetching achieves an extra 10-25 GFLOPS when running under OpenCL.  

To assess the computational performance of the semi-nodal pyramid basis, we implemented a GPU-accelerated DG solver using the OCCA scripting language \cite{medina2014occa}.  The solver is written using OCCA kernels, which may then be expanded to various threading languages for portability across differing architectures.  Numerical experiments suggest that OCCA kernels, translated into CUDA, OpenCL, or OpenMP, perform nearly as well as hand-tuned kernels written directly in the native language \cite{gandham2014gpu}.    The DG solver is run for a fixed number of timesteps on an Nvidia GeForce GTX 980 using CUDA (which we abbreviate as ``Nvidia'') and an AMD Tahiti GPU using OpenCL (which we abbreviate as ``AMD'').  Additionally, we ran the same kernels up to $N=4$ on an Intel Core i7-5960X CPU using OpenMP (which we abbreviate as ``CPU'').  The GFLOPS and estimated effective bandwidth of each kernel are reported in Figure~\ref{fig:advec_cpu}.  We note that the effective bandwidth estimates do not consider caching effects; as a result, the reported numbers may exceed the maximum available device bandwidth.  Table~\ref{table:legend} gives a legend of abbreviations and their respective computational platforms.  

\begin{table}
\centering                                                                                   
\begin{tabular}{| l | l |}
\hline
AMD & AMD Tahiti + OpenCL\\ 
\hline
Nvidia & Nvidia GeForce GTX 980 + CUDA\\ 
\hline
CPU & Intel Core i7-5960X + OpenMP\\
\hline
\end{tabular}
\caption{Legend of abbreviations for different computational architectures.}
\label{table:legend}
\end{table}

The mesh is taken to be a $16\times 16 \times 16$ mesh of hexahedral elements, each of which is then subdivided into 6 pyramids  to produce 24576 elements.  The order is varied from $N=1$ to $N=6$ (order is limited to $N=5$ when using OpenCL, due to the memory limitations on workgroup size), and both GFLOPS and estimated effective bandwidth (averaged over three runs) are reported for the volume and surface kernels in Figures~\ref{fig:advec_vol_gflops}, \ref{fig:advec_surf_gflops}, and \ref{fig:advec_rk_gflops}.  

\begin{figure}
\centering
\subfigure[GFLOPS]{\includegraphics[width=.45\textwidth]{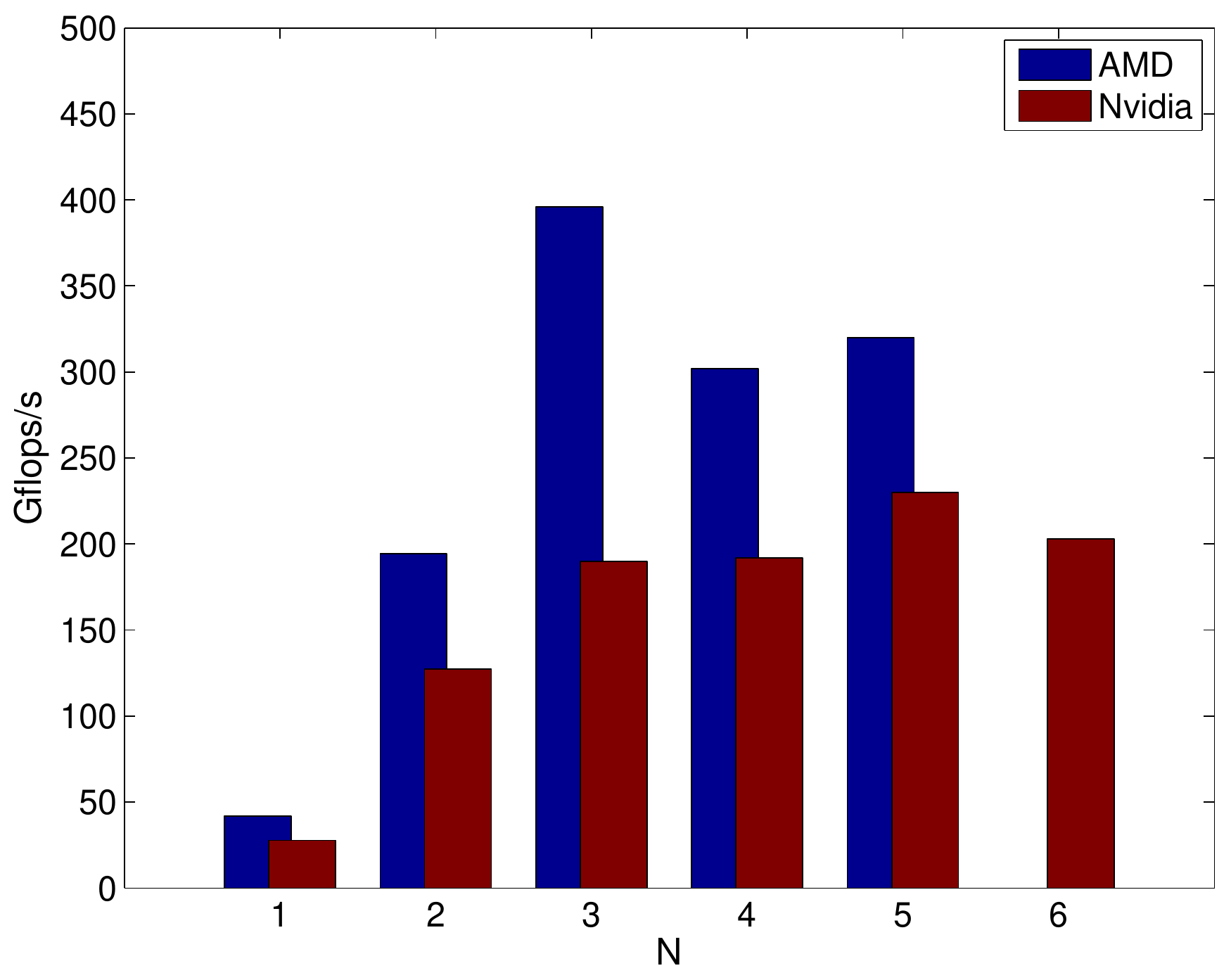}}
\subfigure[Est. bandwidth]{\includegraphics[width=.45\textwidth]{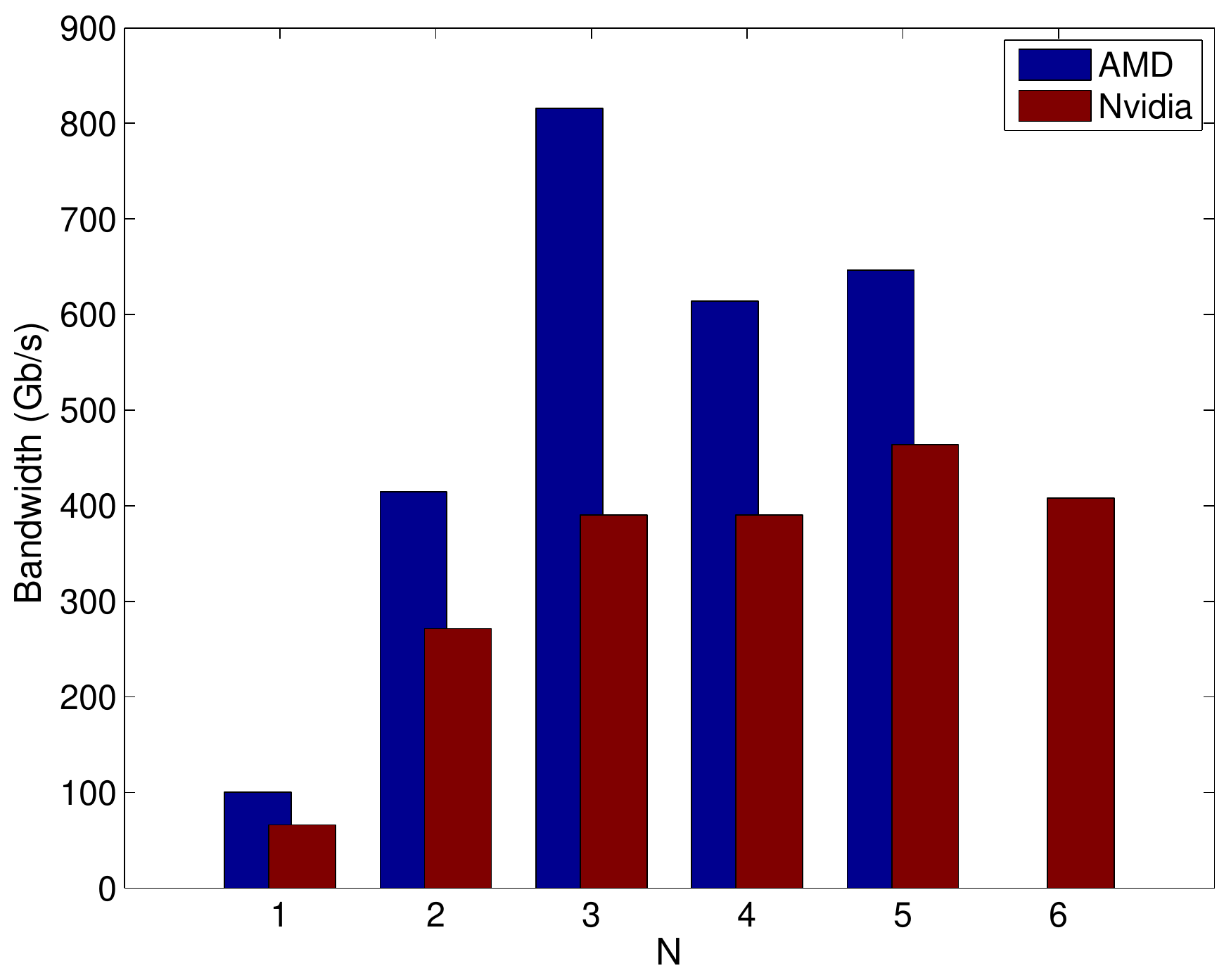}}
\caption{GFLOPS and estimated effective bandwidth for the advection volume kernel under both the AMD and Nvidia setup.}
\label{fig:advec_vol_gflops}
\end{figure}

\begin{figure}
\centering
\subfigure[GFLOPS]{\includegraphics[width=.45\textwidth]{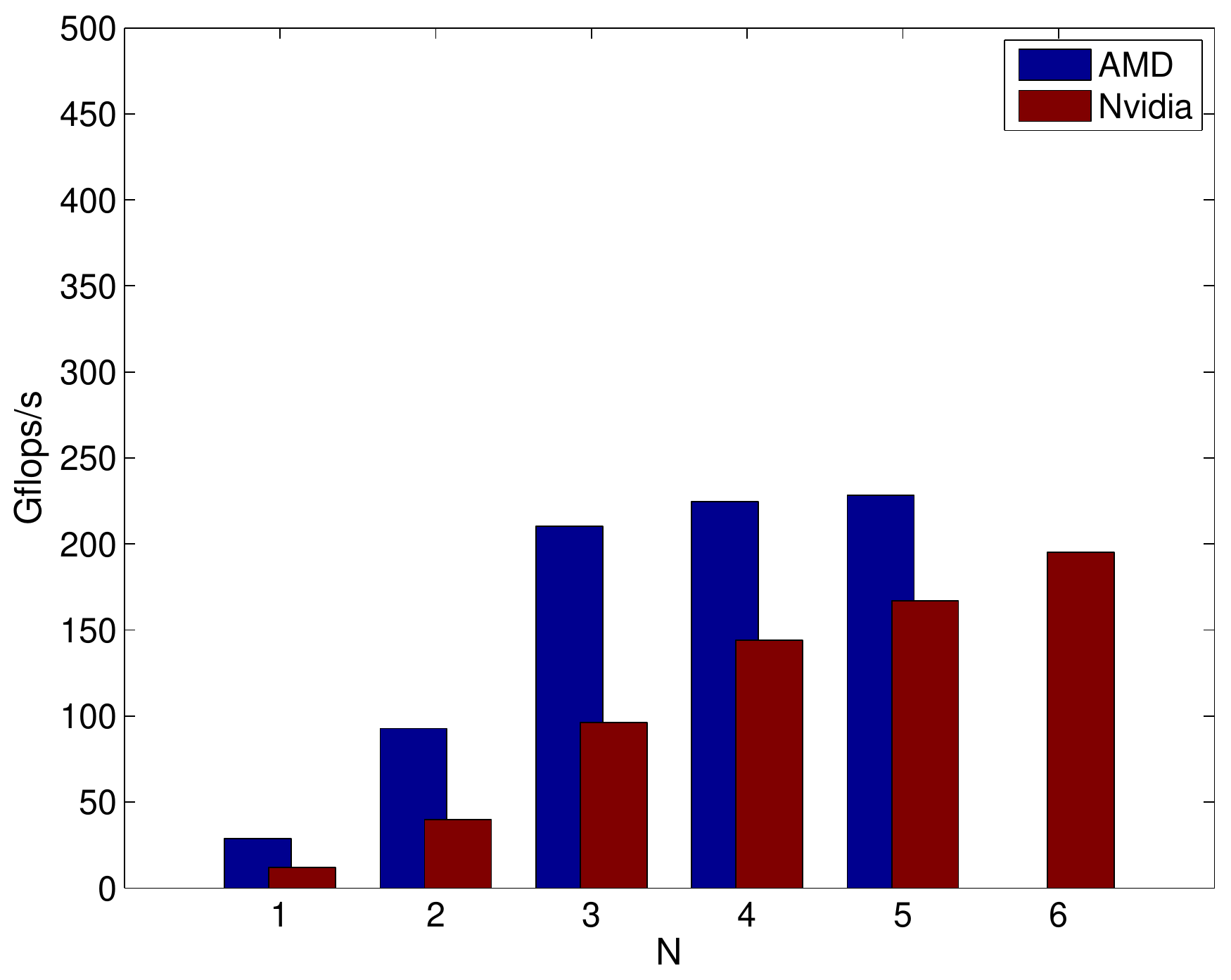}}
\subfigure[Est. bandwidth]{\includegraphics[width=.45\textwidth]{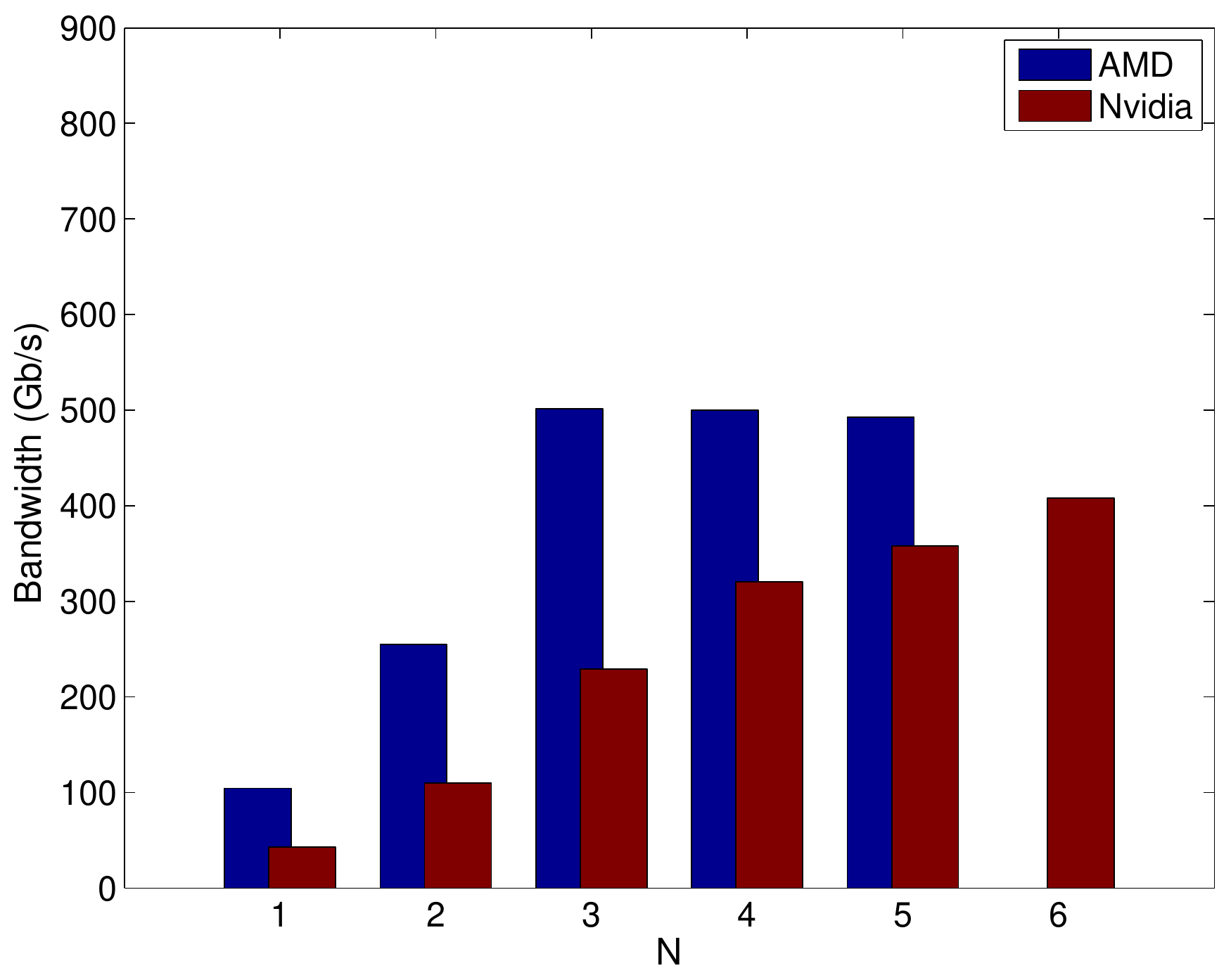}}
\caption{GFLOPS and estimated effective bandwidth for the advection surface kernel under both the AMD and Nvidia setup.}
\label{fig:advec_surf_gflops}
\end{figure}

\begin{figure}
\centering
\subfigure[GFLOPS]{\includegraphics[width=.45\textwidth]{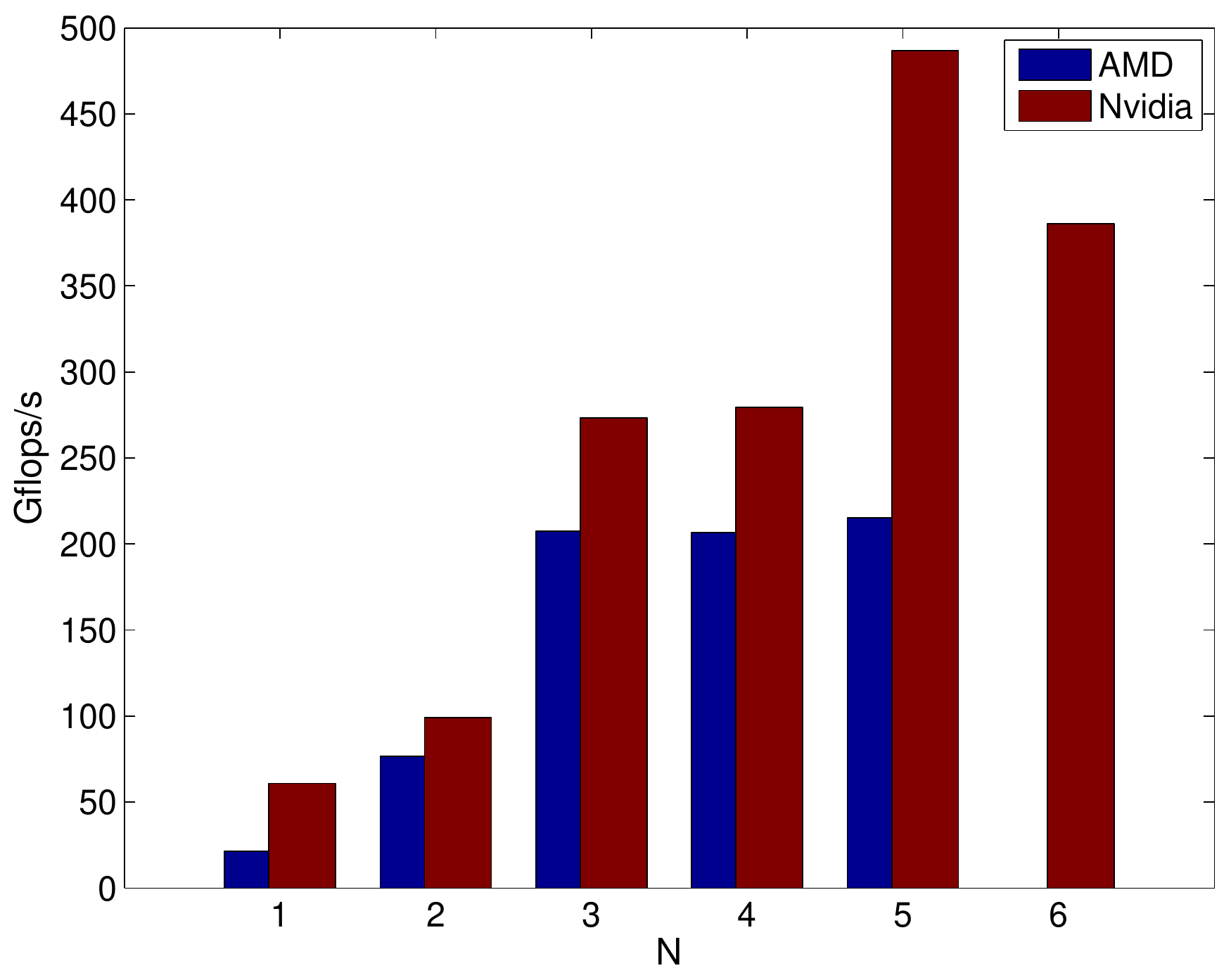}}
\subfigure[Est. bandwidth]{\includegraphics[width=.45\textwidth]{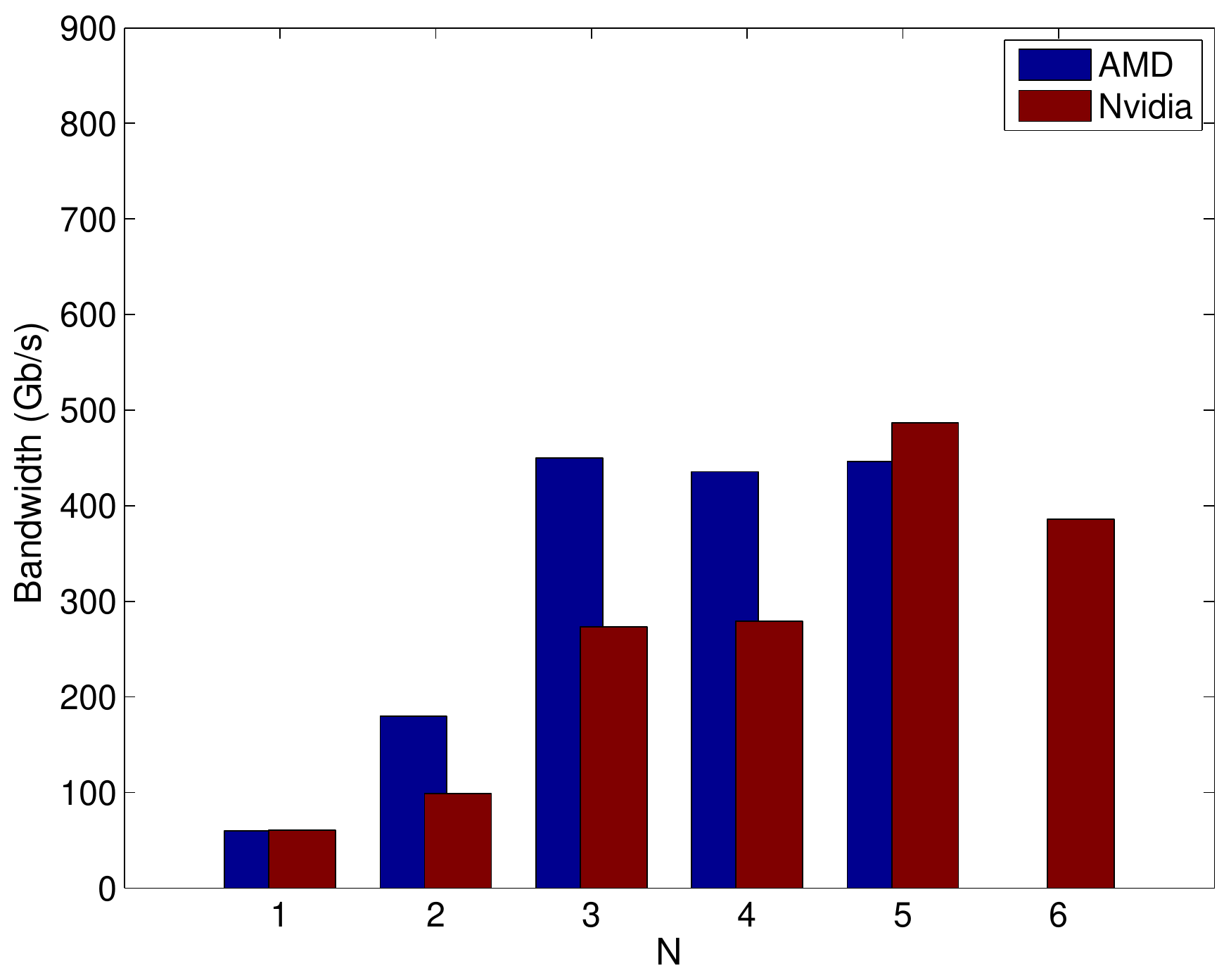}}
\caption{GFLOPS and estimated effective bandwidth for the advection RK update kernel under both the AMD and Nvidia setup.}
\label{fig:advec_rk_gflops}
\end{figure}

We note that the effect of caching is relatively significant in computing estimated bandwidth.  If the bandwidth is estimated without counting operator loads, we get

\begin{figure}
\centering
\subfigure[GFLOPS]{\includegraphics[width=.45\textwidth]{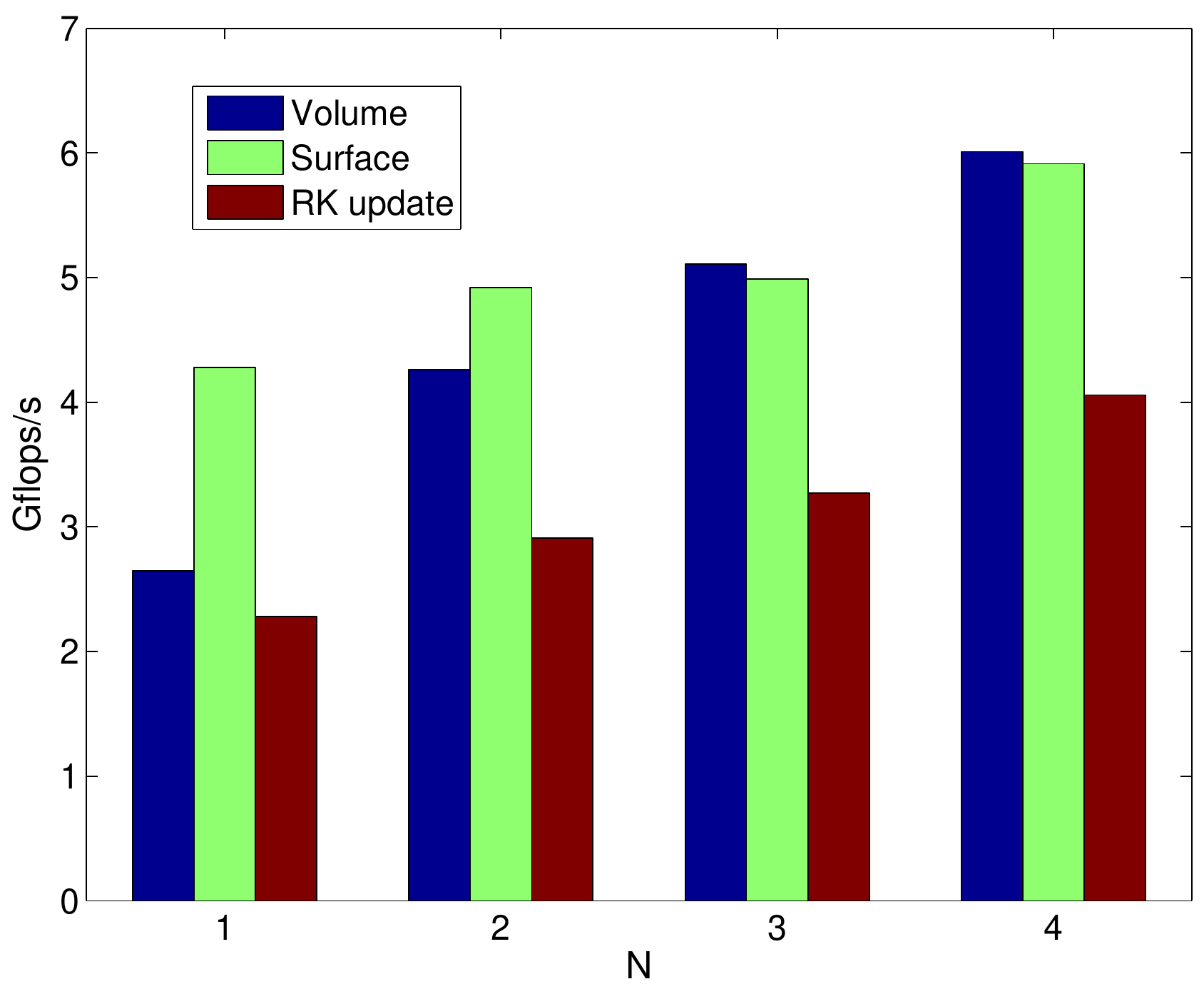}}
\subfigure[Est. bandwidth]{\includegraphics[width=.45\textwidth]{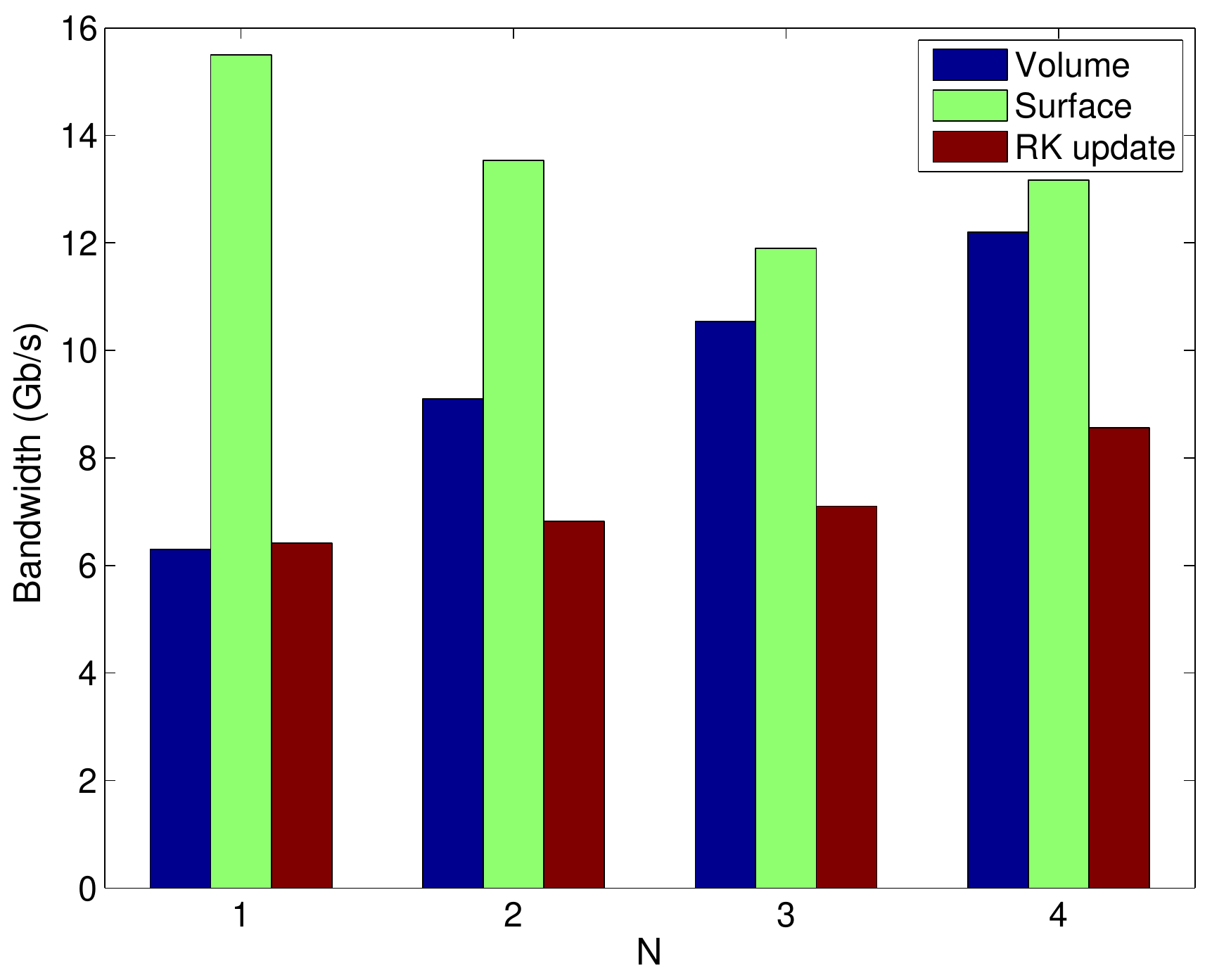}}
\caption{Gflops and estimated effective bandwidth for advection volume, surface, and update kernels under the CPU setup.}
\label{fig:advec_cpu}
\end{figure}

\subsection{Acoustic wave equation}

We consider also the acoustic wave equation on domain $\Omega$ with free surface boundary conditions $p=0$ on $\partial \Omega$.  This may be written in first order form
\begin{align*}
\frac{1}{\kappa}\pd{p}{t}{} + \Div u &= f\\
\rho\pd{\vect{u}}{t}{} + \Grad p &= 0,
\end{align*}
where $p$ is acoustic pressure, $\vect{u}$ is velocity, and $\rho$ and $\kappa$ are density and bulk modulus, respectively, and are assumed to be piecewise constant.  

Let $(p^-,\vect{u}^-)$ denote the solution fields on the face of an element $K$, and let $(p^+,\vect{u}^+)$ denote the solution on the neighboring element adjacent to that face.  Defining the jump of $p$ and the vector velocity $\vect{u}$ componentwise
\[
\jump{\vect{u}} = \vect{u}^+ - \vect{u}^-, \qquad \jump{p} = p^+ - p^-,
\]
the semi-discrete variational formulation for the discontinuous Galerkin method may then be given over an element as
\begin{align*}
\int_K \LRp{\frac{1}{\kappa}\pd{p}{t}{} + \Div \vect{u}} v^- \diff x + \int_{\partial K} \frac{1}{2}\LRp{\vect{n}\cdot \jump{\vect{u}} - \tau_p\jump{p}}v^- \diff x   &= \int_K f v^- \diff x\\
\int_K \LRp{\rho\pd{\vect{u}}{t}{} + \Grad p}v^- + \int_{\partial K} \vect{n}\frac{1}{2}\LRp{\jump{p} - \tau_u \vect{n}\cdot\jump{\vect{u}}}v^- \diff x  &= 0,
\end{align*}
where $\tau_p = 1/\avg{\rho c}$, $\tau_u = \avg{\rho c}$, and $c^2 = \kappa/\rho$ is the speed of sound.  

We discretize by again representing $\vect{u}, v$ using the semi-nodal basis $\phi_{ijk}$.  This converts the variational problem into a system of ODEs
\begin{align*}
\td{p}{t} + \kappa M^{-1}\LRp{\sum_{k=1}^3 S^k u_k +  L^f P_p} &= b,\\
\td{u_k}{t} + \rho M^{-1}\LRp{S^k u_k + L^{f,k} P_{u_k}} &= 0, \quad k = 1,\ldots, 3,
\end{align*}
where $L^f$ is the scalar lift operator, and $L^{f,k}$ is the vector lift operator defined by
\[
L^{f,k}_{ij} = \int_{\partial K} \vect{n}_k \phi_i(x) \phi_j(x) \diff x \approx \sum_{l=1}^{N^f_c}w_l \phi_i(x_l) \phi_j(x_l) 
\]
applied to the penalty terms $P_p, P_{u_i}$, and $S^k$ is the weak derivative matrix defined by
\[
S^k_{ij} = \int_{K} \pd{\phi_j}{x_k} \phi_i(x) \diff x \approx  \sum_{l=0}^{N_c}w_l \pd{\phi_j(x_l)}{x_k} \phi_i(x_l)
\]
for an appropriate $N_c$ point quadrature rule with points $x_l$ and weights $w_l$.

Defining a vector variable $\vect{U} = (p,\vect{u})$, we may write our system of ODEs for the wave equation as
\[
\td{\vect{U}}{t} = A\vect{U}.
\]
The computed spectral radii of the RHS matrix $\rho({\bf A})$ are given in Figure~\ref{fig:wave_CFL} for various mesh sizes $h$ (computed as the ratio of surface area to volume of an element) and a function of the order of approximation $N$.  The spectral radius $\rho({\bf A})$ gives an estimate of the maximum timestep under which an explicit scheme remains stable, and for standard polynomial finite element spaces is proportional to $N^2/h$.  We observe the same behavior numerically for pyramids, and note that the spectral radius shows very good agreement with $2(N+1)(N+3)/3$, which is the $N$-dependent constant in the discrete trace inequality for the pyramid \cite{chan2015hp}.\footnote{Though the meshes used to compute the spectral radii are uniform, randomly perturbing the vertex positions does not change the value of $\rho({\bf A})$ significantly, which was also observed in \cite{bergot2010higher}.}  

\begin{figure}
\centering
\subfigure[$h^{-1}$ vs $\rho({\bf A})$]{\includegraphics[width=.475\textwidth]{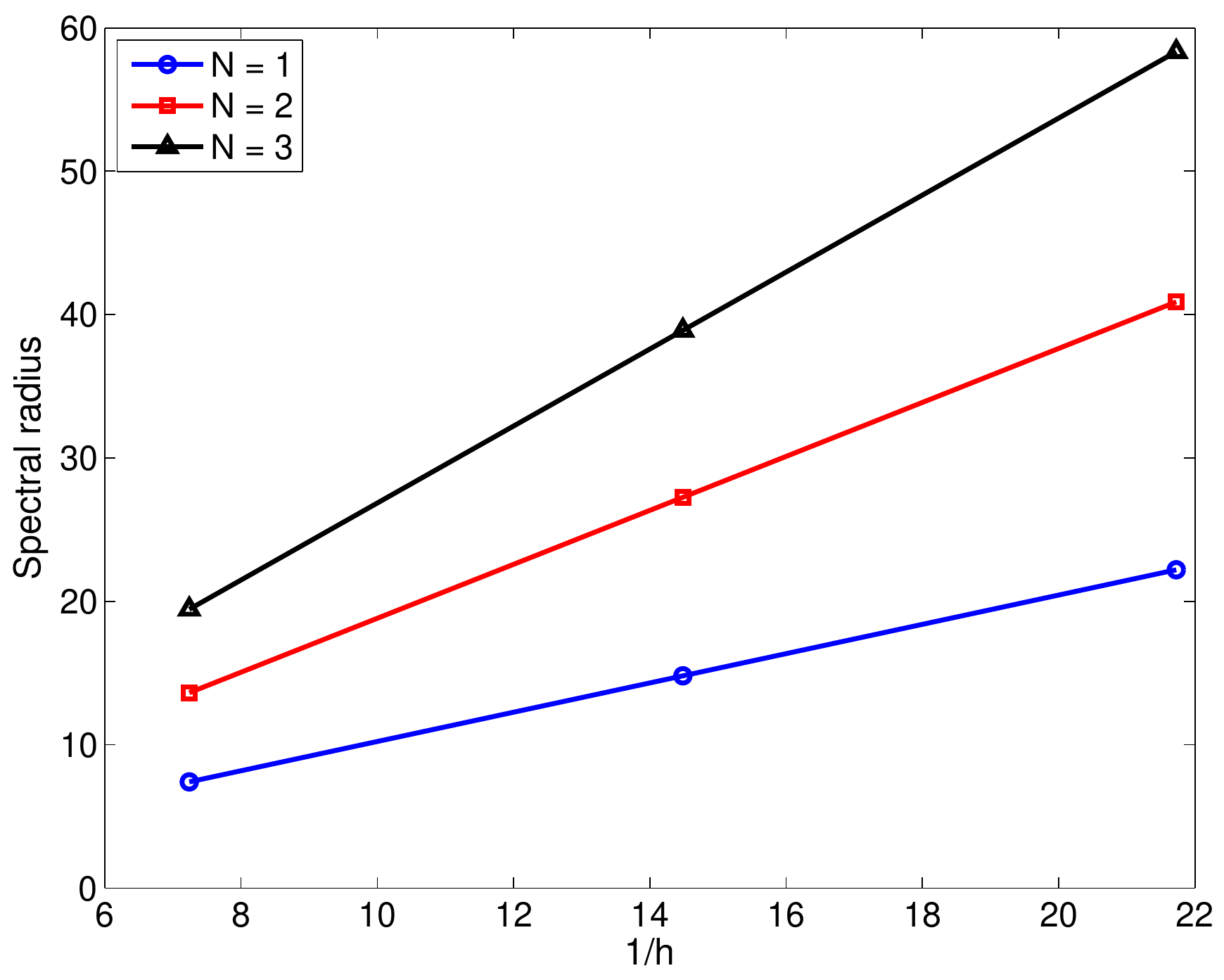}}
\subfigure[$N$ vs $\rho({\bf A})$]{\includegraphics[width=.475\textwidth]{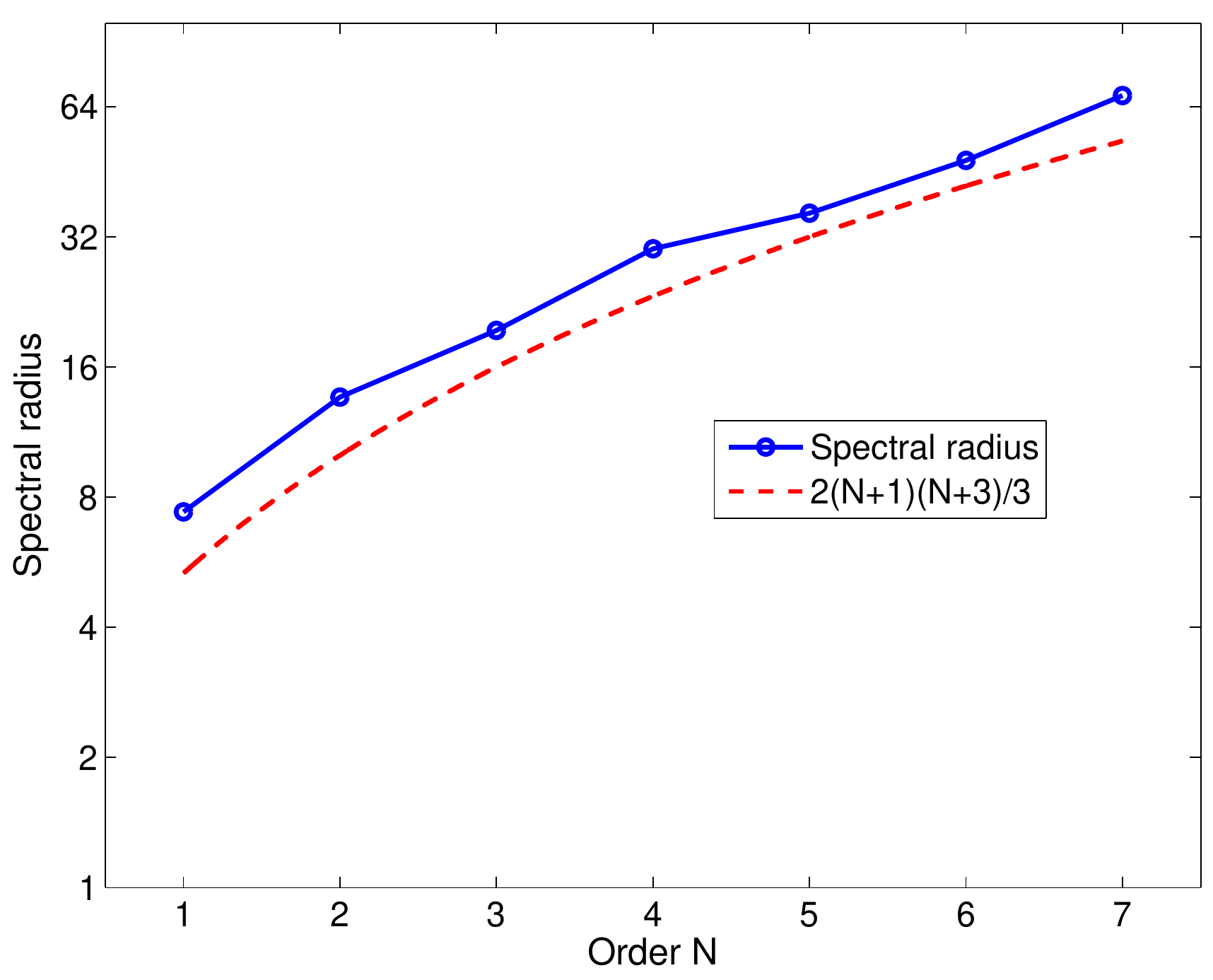}}
\caption{Ratio of numerically computed spectral radii $\rho({\bf A})$, plotted against both mesh size $h$ and $2(N+1)(N+3)/3$.}
\label{fig:wave_CFL}
\end{figure}

We also report numerical convergence rates in Figure~\ref{fig:wave_rates} for the resonant cavity solution 
\[
p(x,y,z,t) = \cos\LRp{{\pi x}/{2}}\cos\LRp{{\pi y}/{2}}\cos\LRp{{\pi z}/{2}}\cos\LRp{{\sqrt{3}\pi t}/{2}}.
\]
over the bi-unit cube $[-1,1]^3$.  Meshes are again constructed by subdividing the cube into $K_{\rm 1D} \times K_{\rm 1D}\times K_{\rm 1D}$ hexahedra, which are then each subdivided into 6 pyramids.  Pyramid vertex positions are perturbed to ensure $J$ is non-constant in each element.  

  It was confirmed in \cite{bergot2010higher} that the rational basis (\ref{eqn:rationalbasis}) achieves optimal $O(N+1)$ rates of convergence for both the $L^2$ and dispersion error.  Since the basis defined by $\phi_{ijk}$ spans the same approximation space as that of (\ref{eqn:rationalbasis}), the numerical errors and convergence rates are also of optimal order, and we observe both optimal rates of convergence in $h$ and exponential convergence in $N$.  The errors are computed in double precision on the GPU; when using single precision, the convergence rates behave similarly, but $L^2$ errors stall at around $10^{-6}$ due to finite precision effects.  

\begin{figure}
\centering
\subfigure[$h$-convergence]{\includegraphics[width=.475\textwidth]{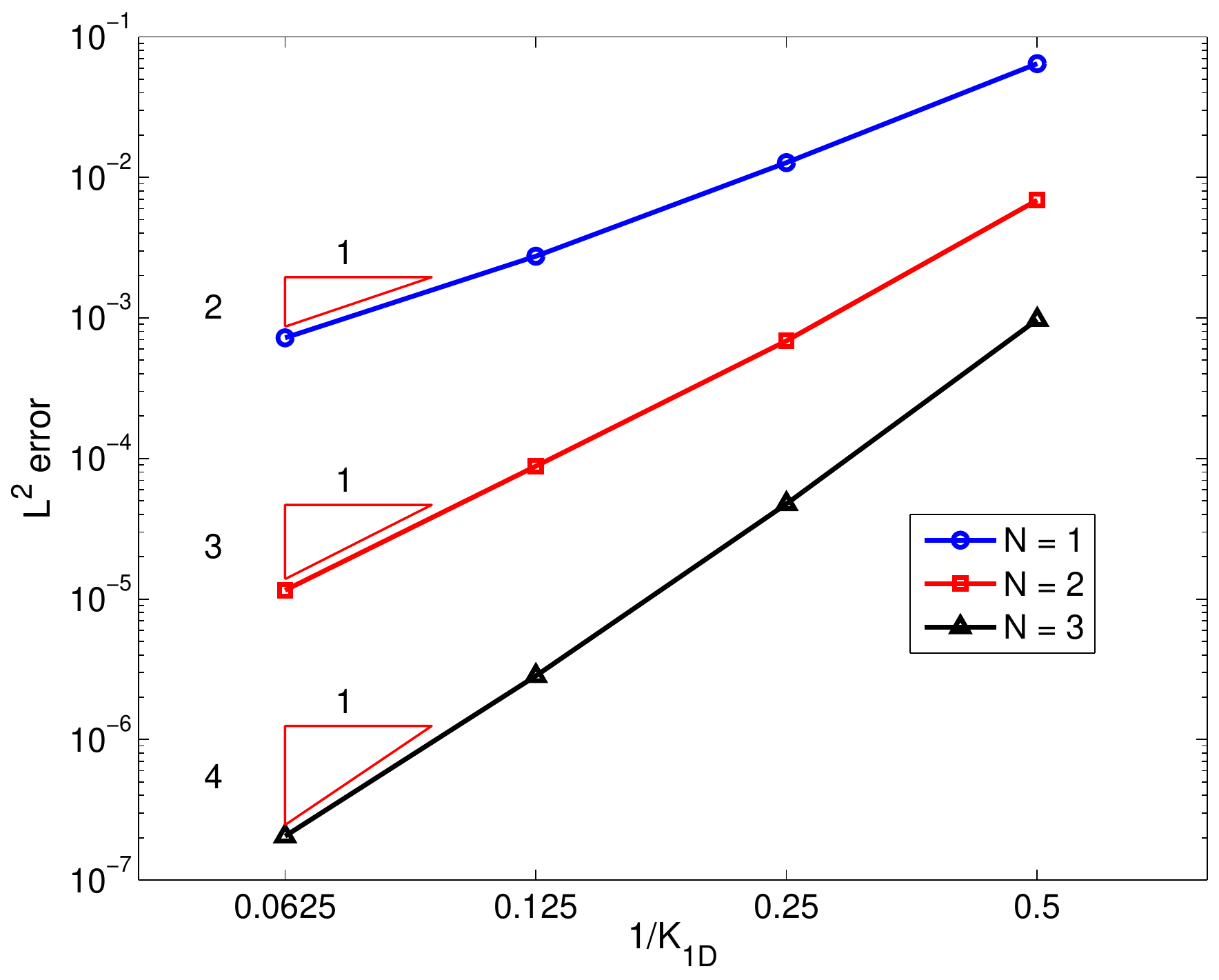}}
\subfigure[$N$-convergence]{\includegraphics[width=.475\textwidth]{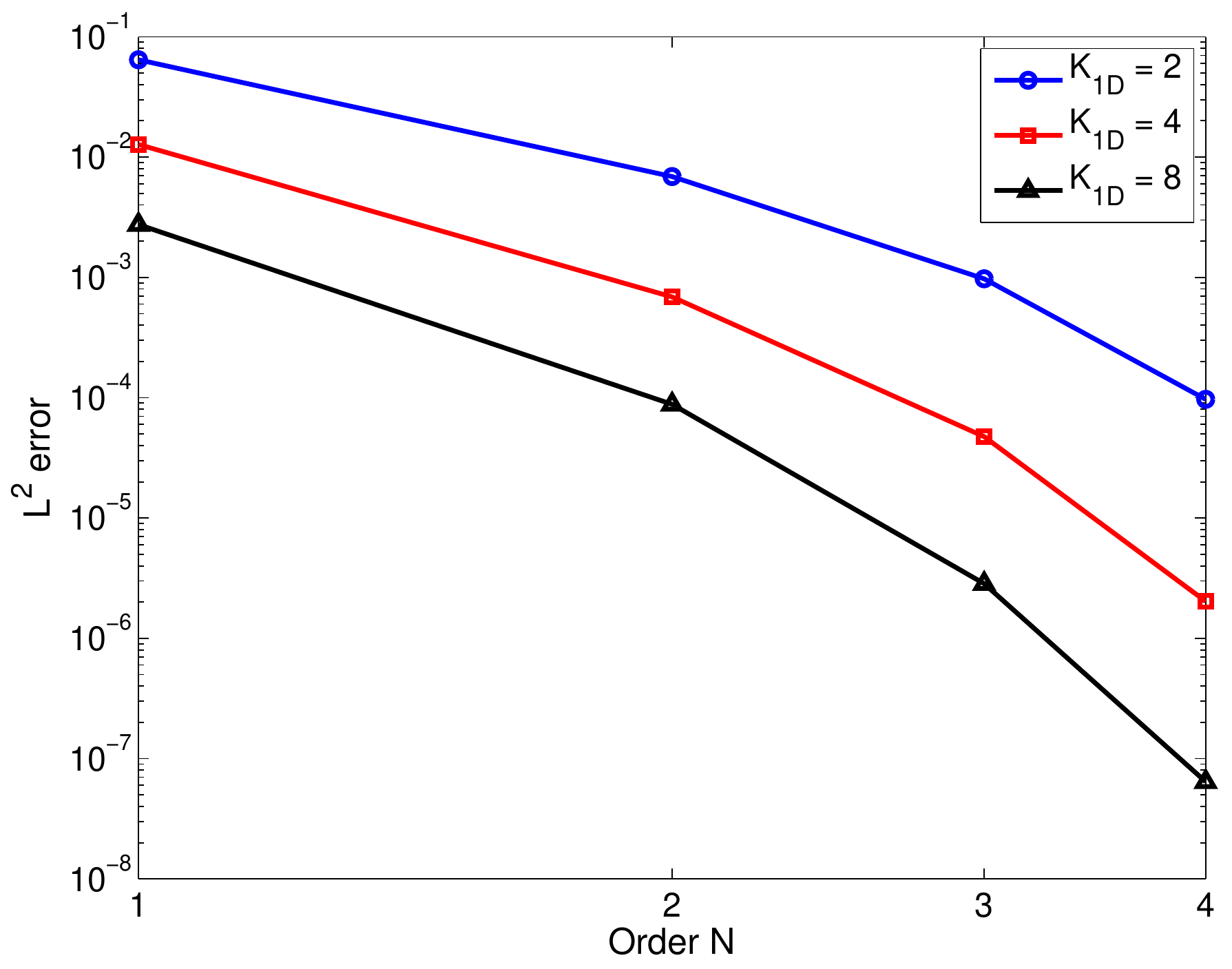}}
\caption{Computed $L^2$ errors for various orders $N$ and mesh sizes, with optimal rates of $h$-convergence for reference.}
\label{fig:wave_rates}
\end{figure}

We present GFLOPS and estimated effective bandwidth for the volume, surface, and RK update kernel in Figures~\ref{fig:wave_vol_gflops}, \ref{fig:wave_surf_gflops}, and \ref{fig:wave_rk_gflops}.  While the estimated effective bandwidth for acoustic wave kernels is similar to that of the advection kernels, the GFLOPS have increased by a factor of 2-4, due to the reuse of derivative and interpolation operators over multiple field variables.  This may be further confirmed by examining the estimated effective bandwidth without counting operator loads, in which case the reported bandwidth decreases by roughly an order of magnitude.  

\begin{figure}
\centering
\subfigure[GFLOPS]{\includegraphics[width=.45\textwidth]{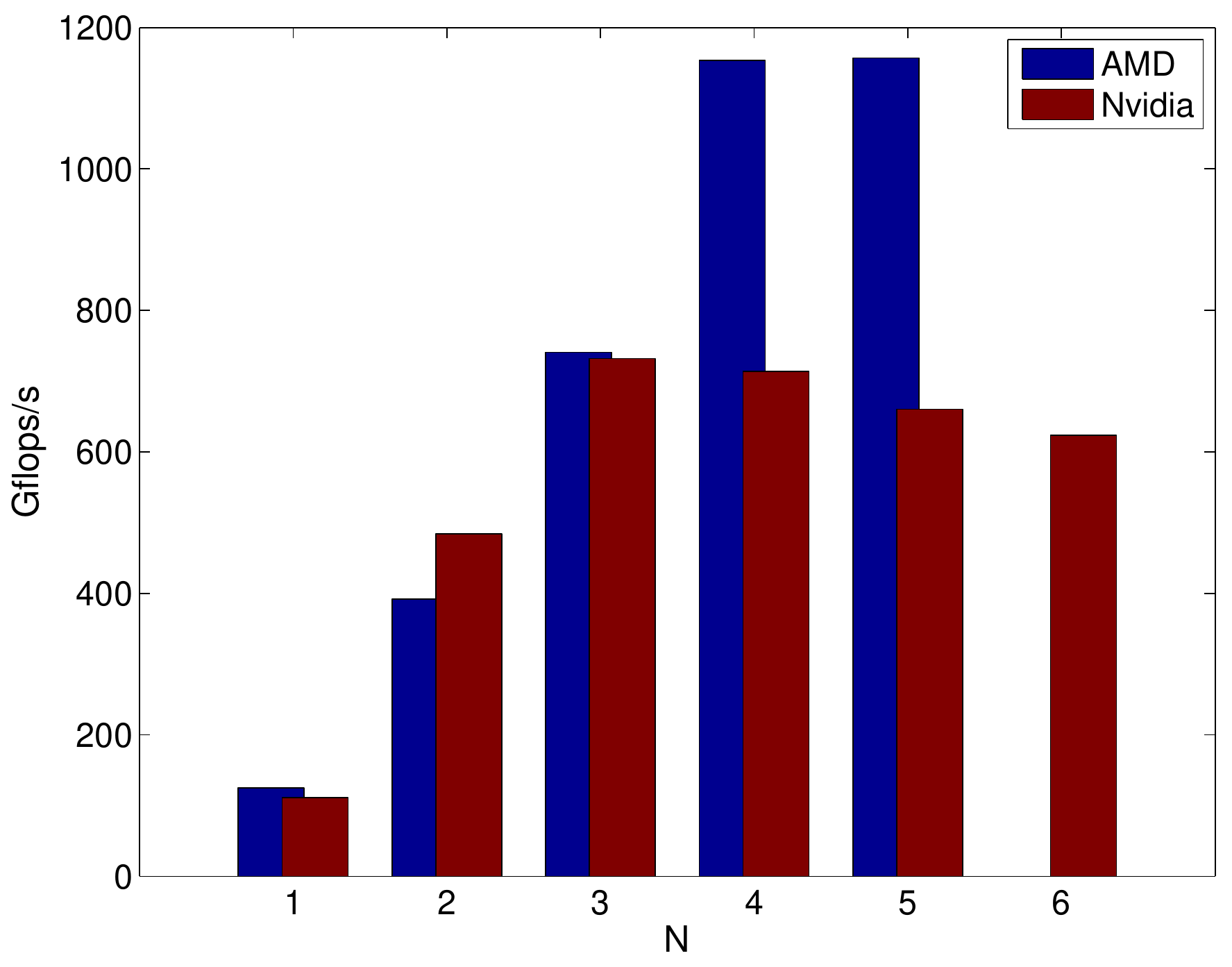}}
\subfigure[Est. bandwidth]{\includegraphics[width=.45\textwidth]{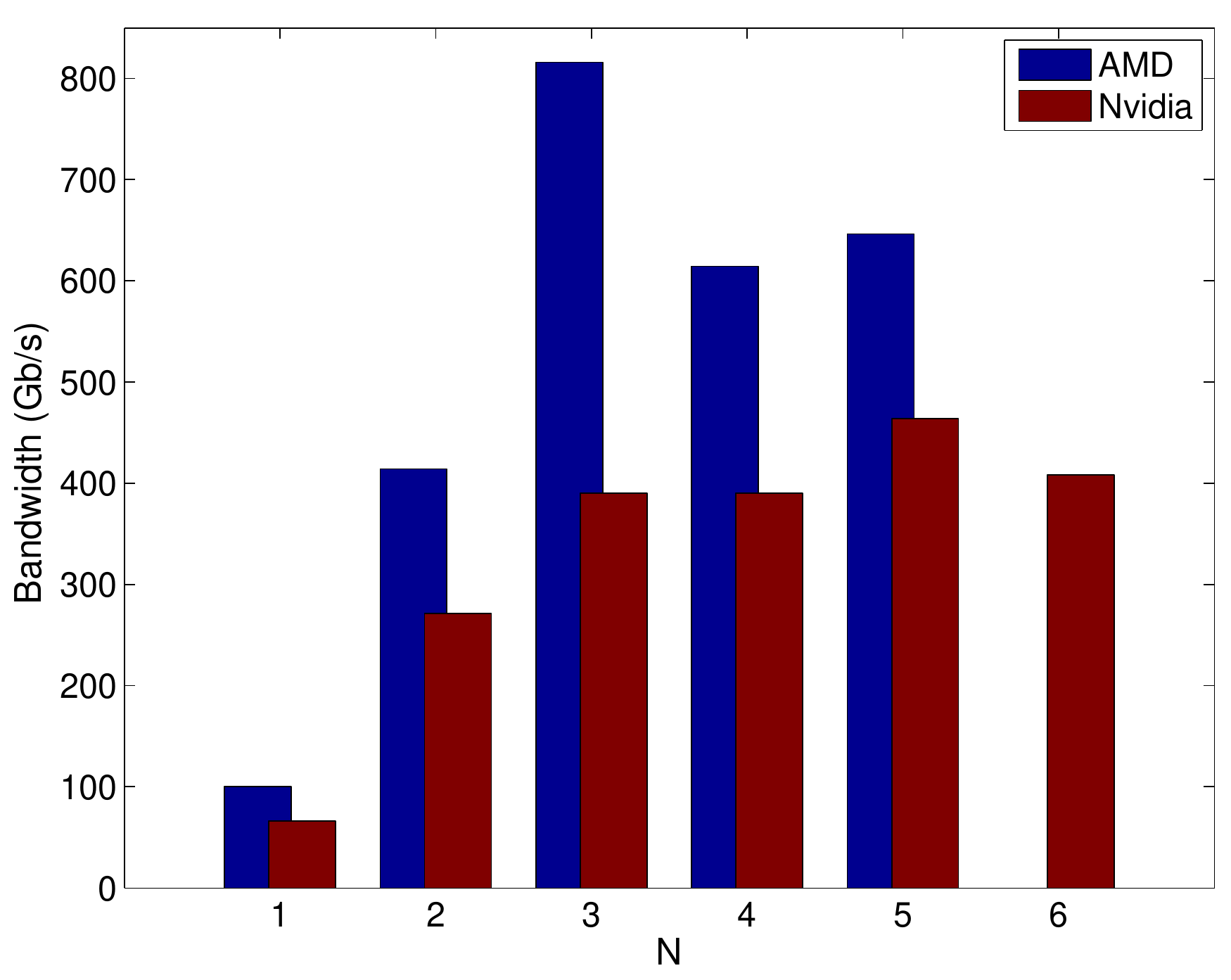}}
\caption{GFLOPS and estimated effective bandwidth for the wave volume kernel under both the AMD and Nvidia setup.}
\label{fig:wave_vol_gflops}
\end{figure}

\begin{figure}
\centering
\subfigure[GFLOPS]{\includegraphics[width=.45\textwidth]{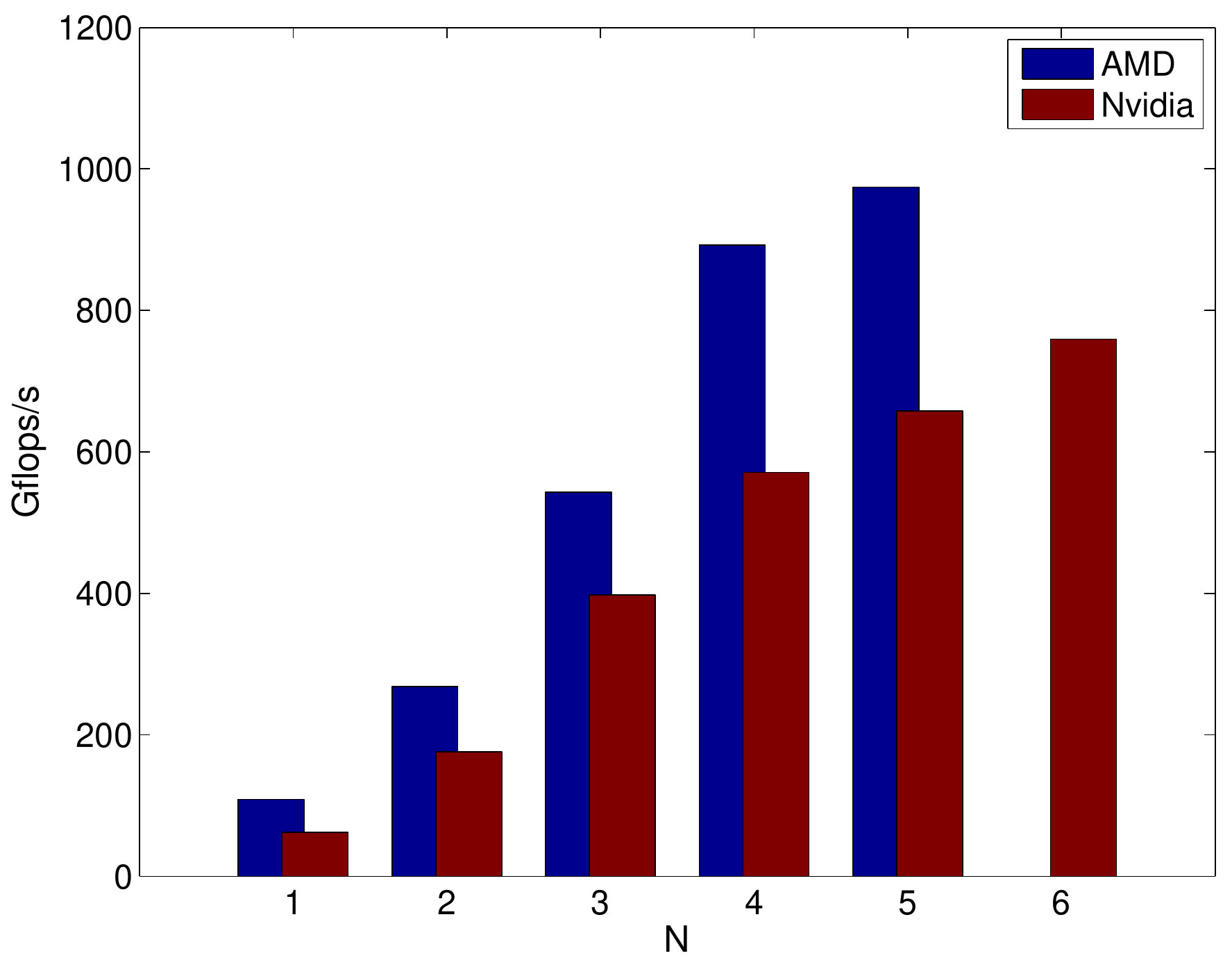}}
\subfigure[Est. bandwidth]{\includegraphics[width=.45\textwidth]{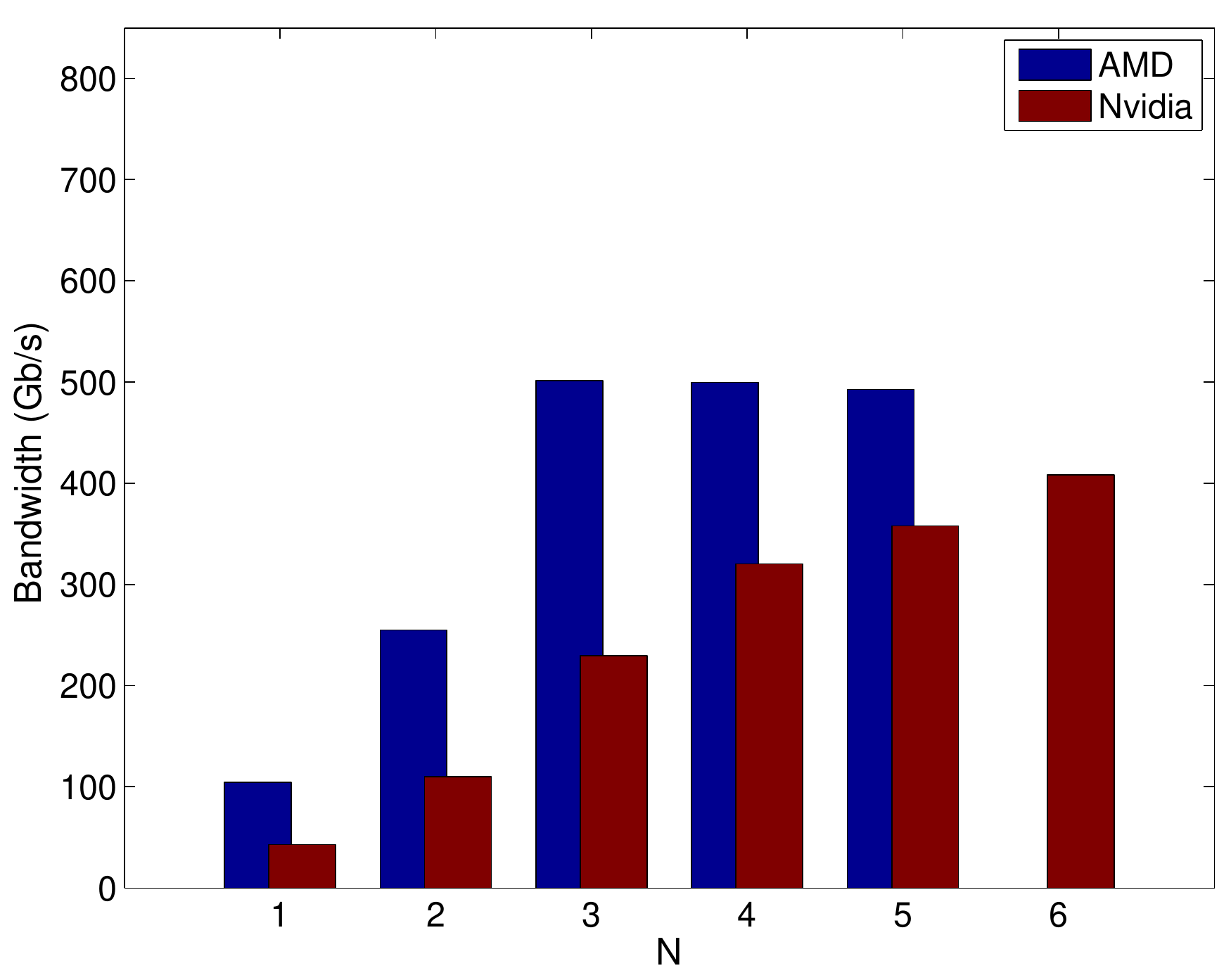}}
\caption{GFLOPS and estimated effective bandwidth for the wave surface kernel under both the AMD and Nvidia setup.}
\label{fig:wave_surf_gflops}
\end{figure}

\begin{figure}
\centering
\subfigure[GFLOPS]{\includegraphics[width=.45\textwidth]{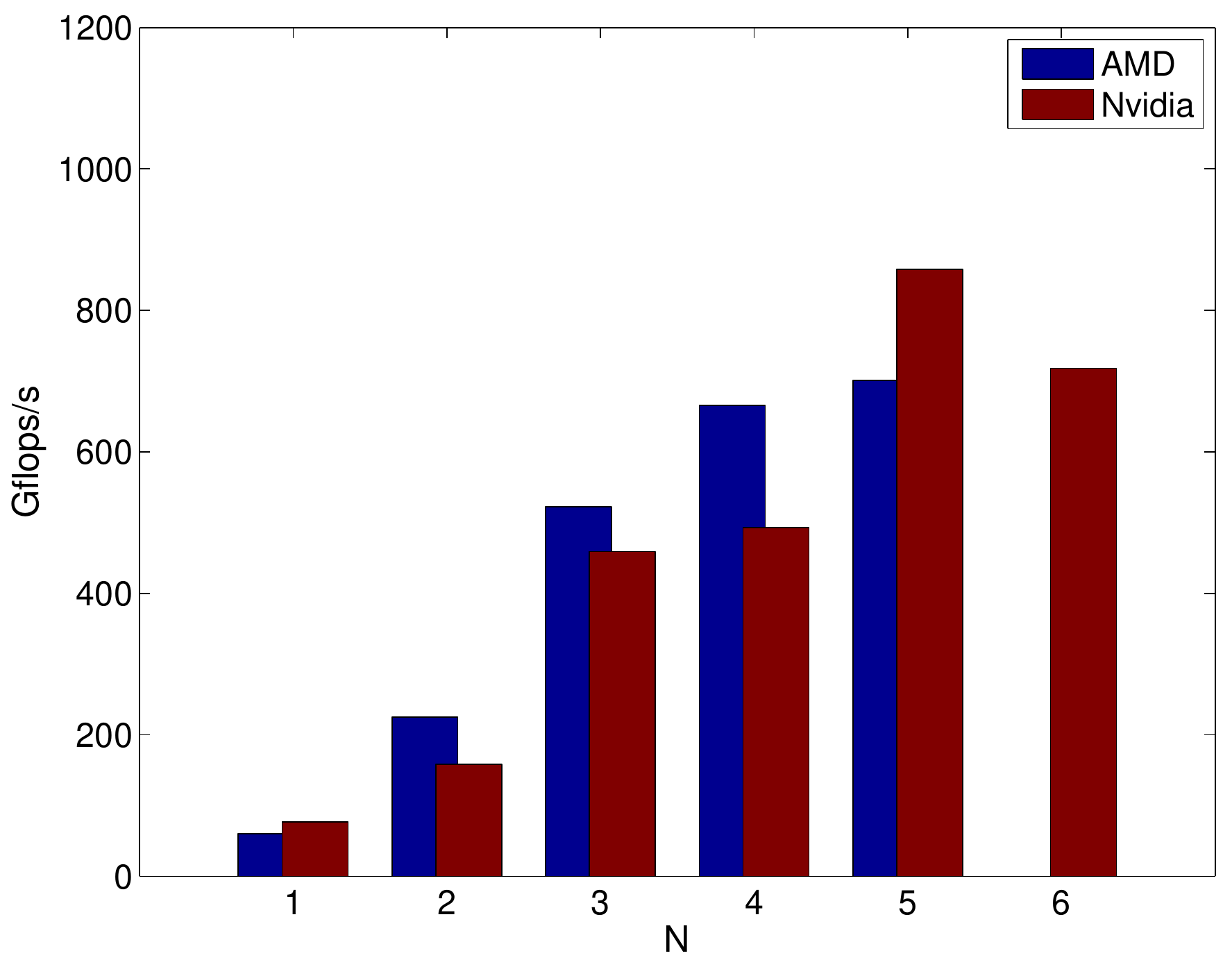}}
\subfigure[Est. bandwidth]{\includegraphics[width=.45\textwidth]{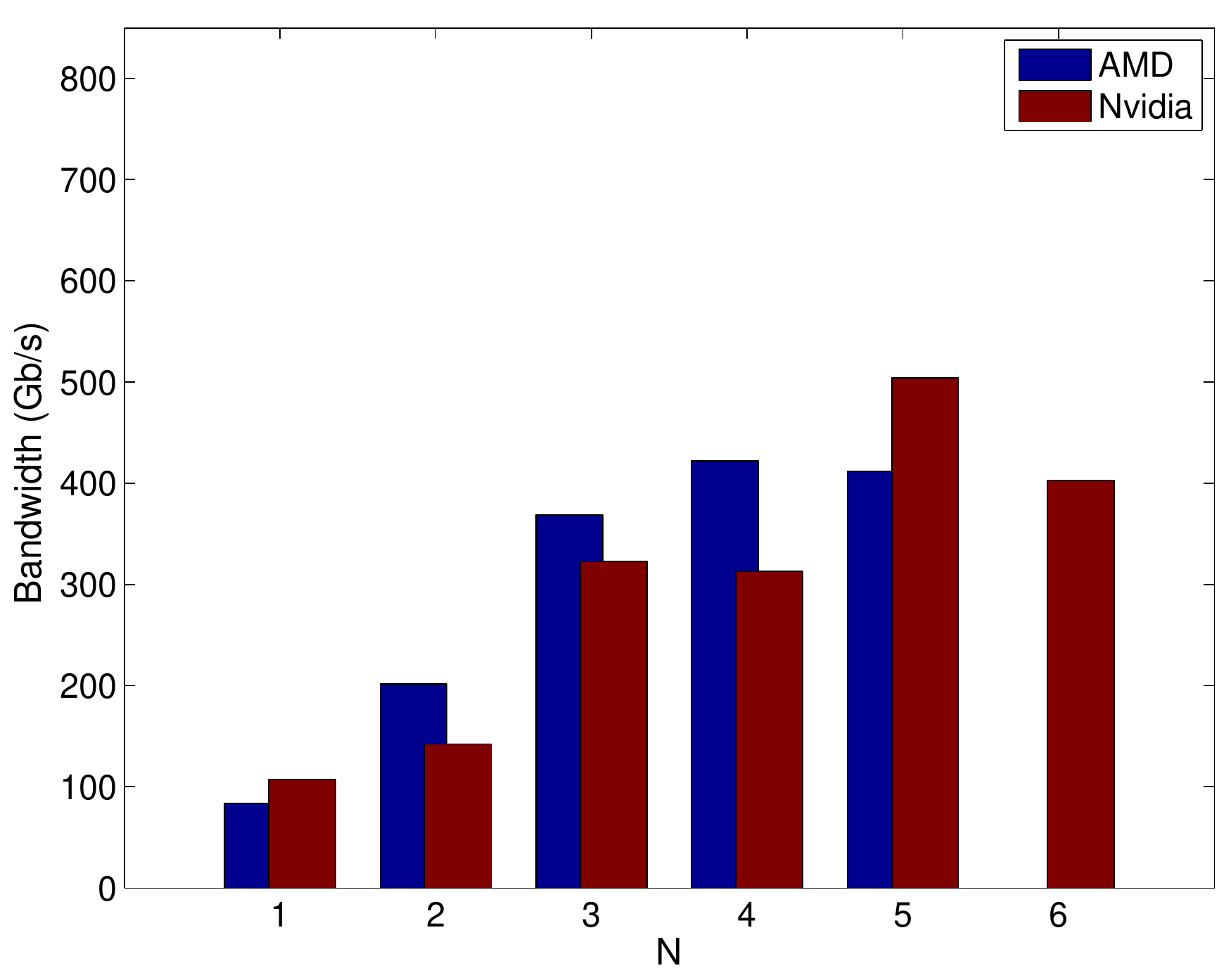}}
\caption{GFLOPS and estimated effective bandwidth for the wave RK update kernel under both the AMD and Nvidia setup.}
\label{fig:wave_rk_gflops}
\end{figure}

Figure~\ref{fig:wave_cpu} shows GFLOPS and estimated effective bandwidth on an Intel Core i7-5960X CPU using OpenMP.  Again, the GFLOPS of acoustic wave kernels increase while the estimated effective bandwidth remains roughly the same as that of the advection kernels, though the increase is not as pronounced as the increase in GFLOPS from advection to the acoustic wave equation on the GPU.  

\begin{figure}
\centering
\subfigure[GFLOPS]{\includegraphics[width=.45\textwidth]{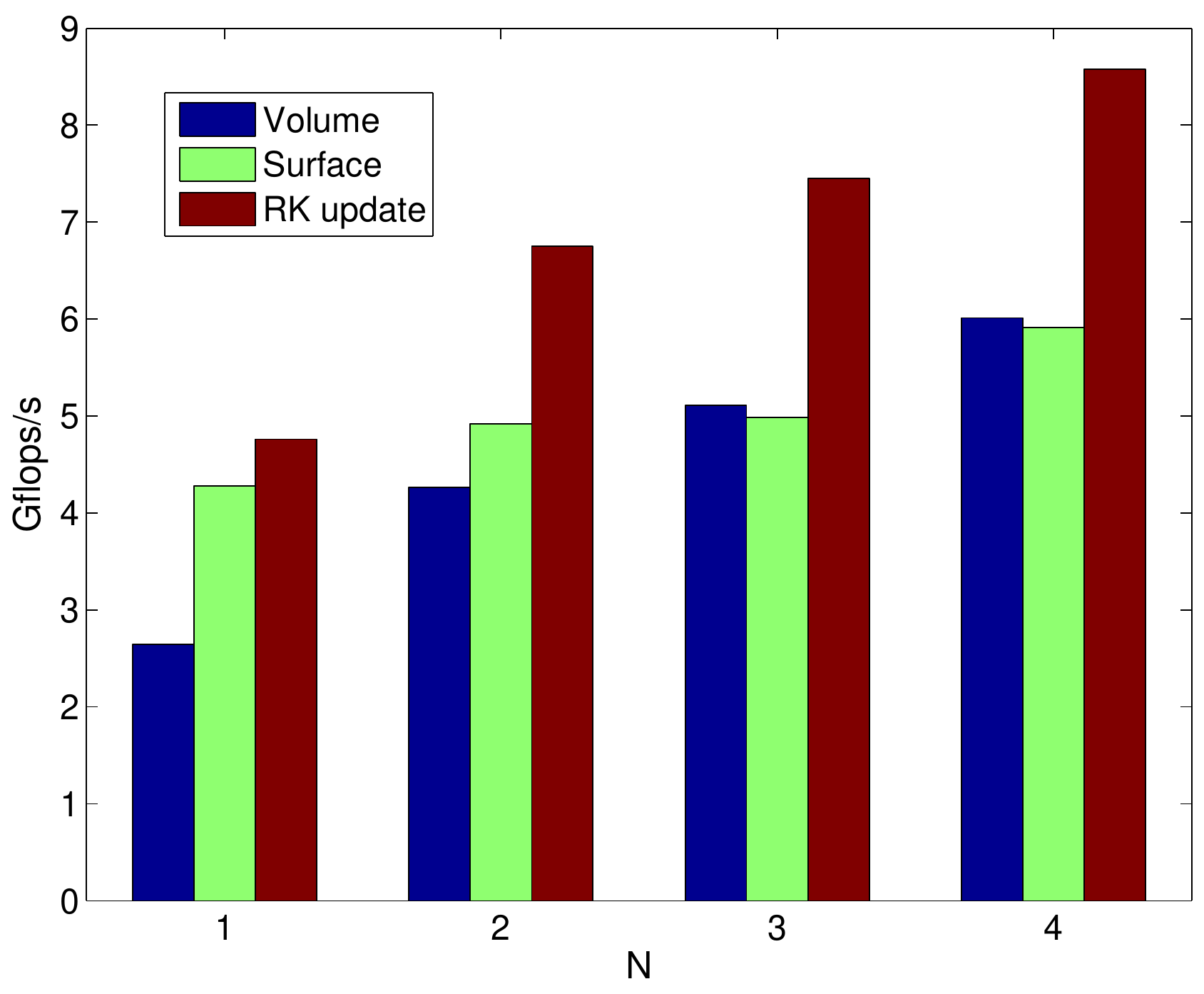}}
\subfigure[Est. bandwidth]{\includegraphics[width=.45\textwidth]{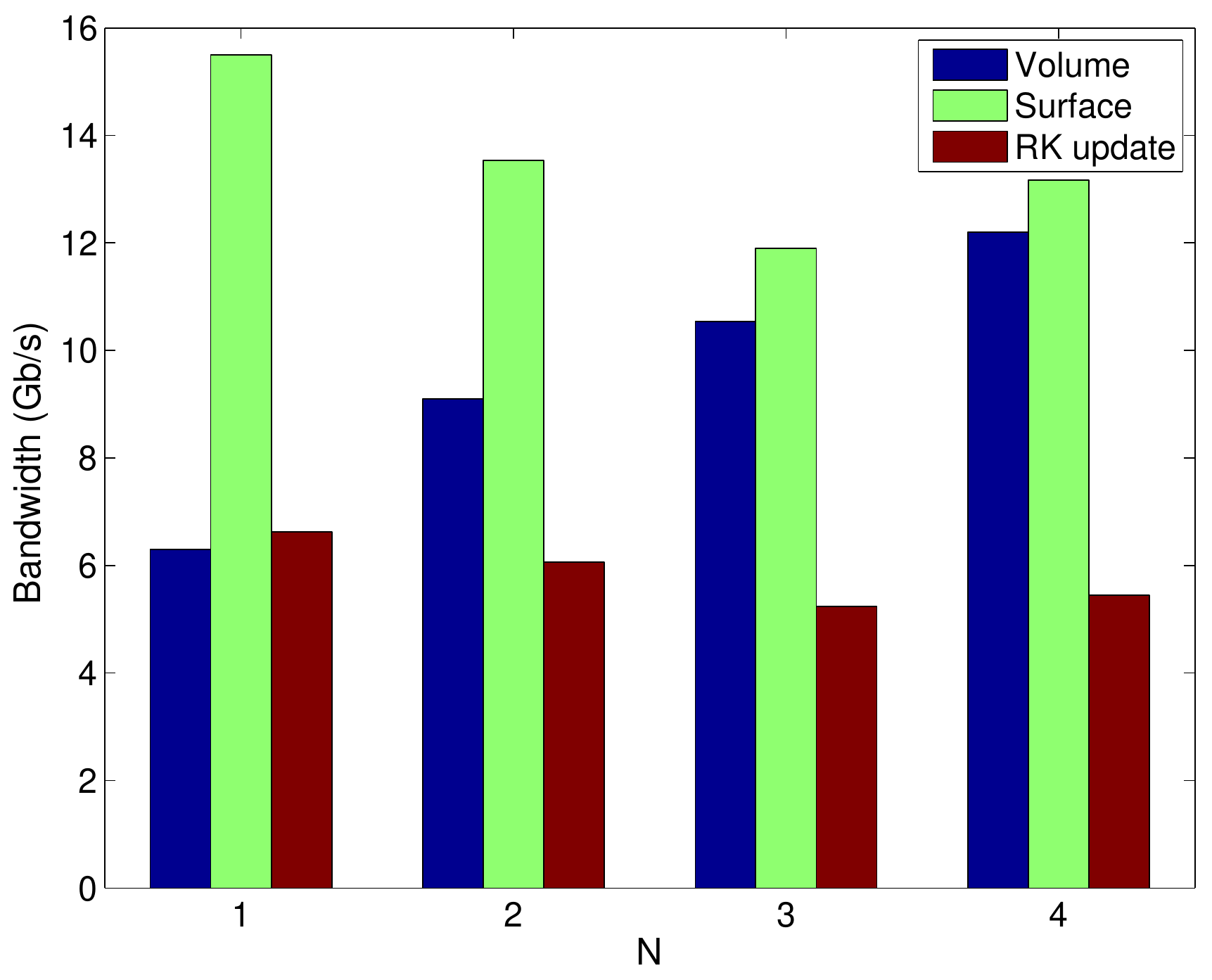}}
\caption{Gflops and estimated effective bandwidth for wave volume, surface, and update kernels under the CPU setup.}
\label{fig:wave_cpu}
\end{figure}

\subsection{Computational improvements}
\label{sec:improvements}

Despite the reported GFLOPS and estimated effective bandwidth reported for pyramids above, it is possible to improve the efficiency of DG on pyramids further.

\subsubsection{Nodal basis functions}

In the above discussions, we discretize by taking the orthogonal (modal) basis $\phi_{ijk}$ defined in Lemma~\ref{lemma:orth}).  However, switching to a nodal discretization using Lagrange basis functions at $N_p$ distinct points on the pyramid requires only a small modification in the application of the mass matrix inverse.  Instead of inverting a diagonal matrix, may be inverted by a change of basis from a nodal to the orthogonal semi-nodal basis.  
%\[
%M^{-1} = V^{-T} \hat{M}^{-1} V^{-1},
%\]
%where $V^{-T}$ is the Vandermonde matrix whose columns consist of values of the $j$th basis function evaluated at the $N_p$ nodal points.  
It is often desirable to define nodal basis functions at strong interpolation points with respect to the Lebesgue constant, which is present in upper bounds on the interpolation error in the maximum norm.  Additionally, for meshes containing multiple element types, it is convenient to choose nodal points on the triangular faces with $(N+1)(N+2)/2$ points on the triangular faces and $(N+1)^2$ points on the quadrilateral face such that the distribution on those faces matches the distribution on the faces of either hexahedra or tetrahedra.  A survey of various nodal points for the pyramid with both low Lebesgue constant and appropriate nodal distributions on faces is given in \cite{chan2014comparison}.  

If the triangular and quadrilateral faces of all elements in a mesh share the same symmetric nodal distribution, conformity under continuous Galerkin methods may be enforced simply by matching the nodal degrees of freedom on the faces of adjacent elements, and the computation of surface integrals may be simplified.  In particular, for discontinuous Galerkin methods on vertex-mapped pyramids, surface integrals may be computed using only nodal degrees of freedom on a face and face mass matrices.  For triangular faces, the face mass matrices are scalings of the reference nodal face mass matrix, and the quadrilateral face mass matrix may be expressed as the Kronecker product of separable scalings of 1D nodal mass matrices.  As a result, the use of nodal mass matrices is more efficient and requires less memory than the use of quadrature for the computation of surface integrals.  Since the computation of surface integrals is the dominant cost for low order discontinuous Galerkin methods, nodal methods are often observed to be more efficient for moderate values of $N$ \cite{bergot2013higher, hesthaven2007nodal}.  For larger values of $N$, the ratio of interior degrees of freedom to surface degrees of freedom increases, and the cost of computing volume integrals becomes dominant.  

\subsubsection{Volume kernel evaluation}

While the computation of surface integrals may also be performed using mass matrices under a nodal basis, volume integrals must still be computed using quadrature due to the the rational nature of the mapping.  Despite the reported GFLOPS and estimated effective bandwidth, the cost of the volume kernel becomes a limiting factor at high orders.  Figure~\ref{fig:runtime} shows reported runtimes and percentage of total runtimes\footnote{The percentage of total runtimes are averages of the percentage of total runtime for advection kernels and percentage of total runtime for wave kernels.} for volume, surface, and update kernels on an Nvidia GeForce GTX 980 over 100 timesteps; at $N > 2$, the volume kernel becomes the dominating bottleneck due to use of tensor product cubature rules for the pyramid, which results in an $O(N^6)$ cost in applying derivative operators, in contrast to $O(N^4)$ cost of surface cubature and applying interpolation operators in the surface and update kernels, respectively.  
\begin{figure} 
\centering
\subfigure[Kernel runtimes]{\includegraphics[width=.45\textwidth]{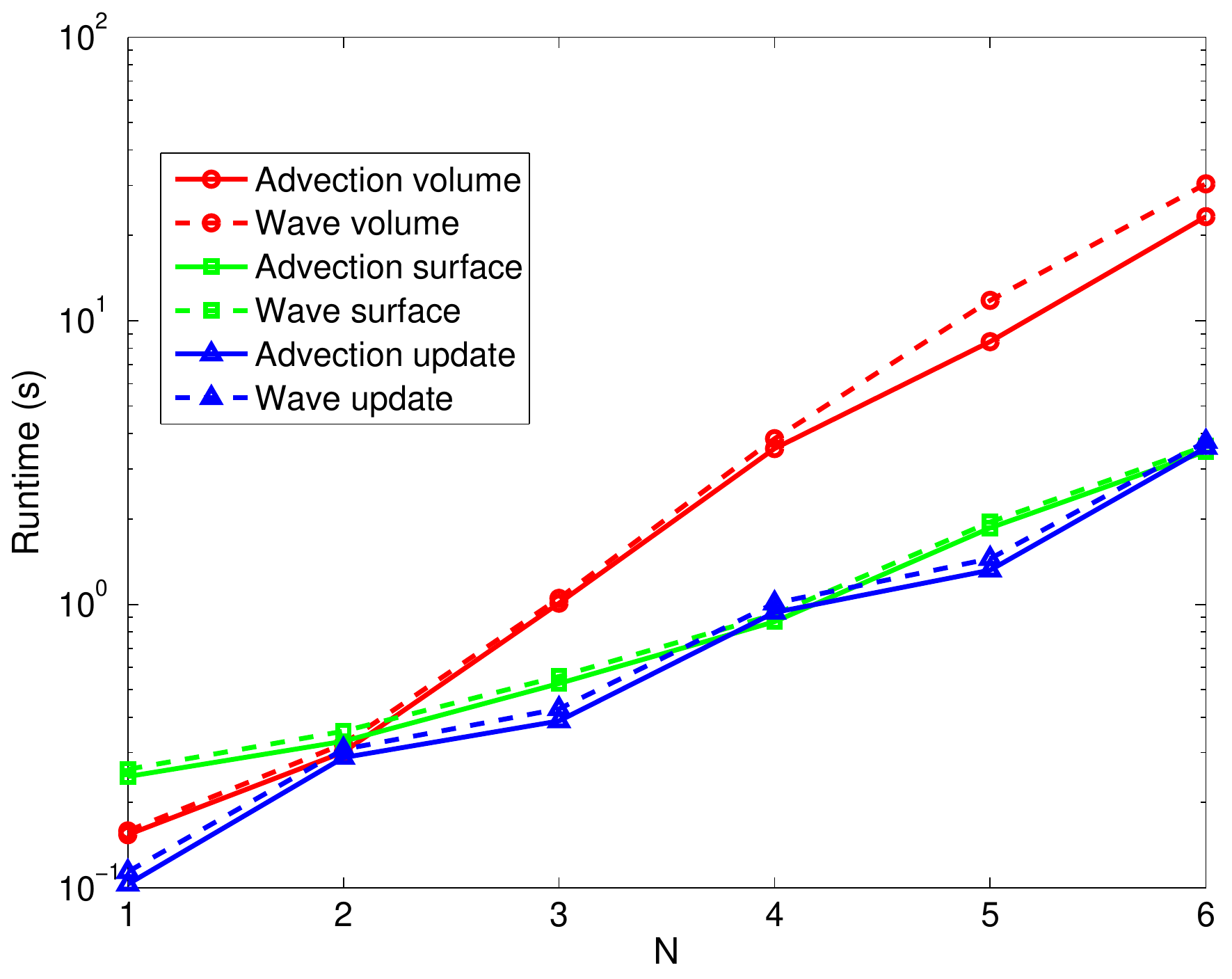}}
\subfigure[Percentage of total runtime]{\includegraphics[width=.45\textwidth]{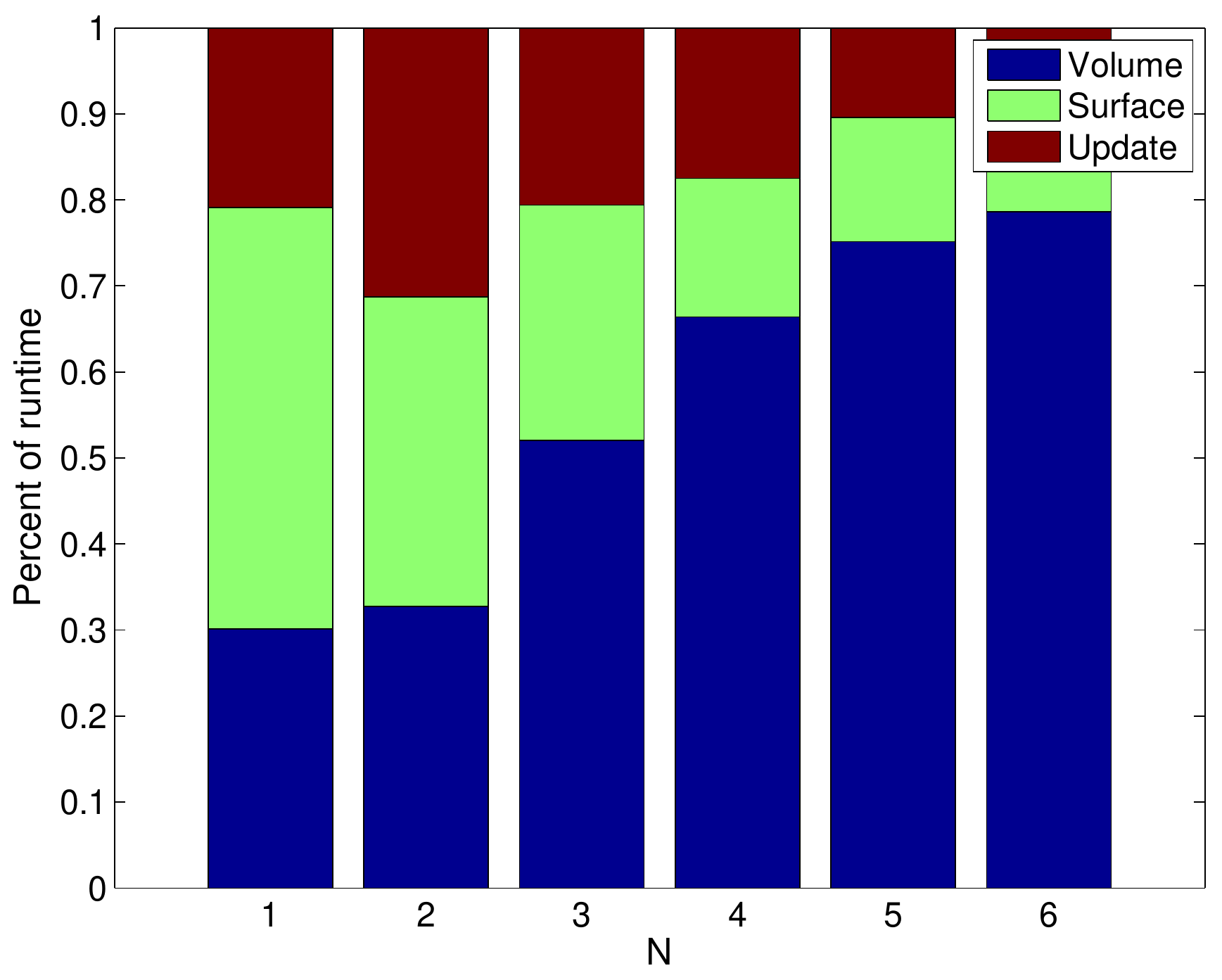}}
\caption{Runtimes and average percentage of total runtime for each individual kernel.}
\label{fig:runtime}
\end{figure}

%Likewise, we note that the majority of estimated bandwidth is dedicated to the loading of Vandermonde (interpolation) and derivative matrices.  We compute the bandwidth with and without loading interpolation and derivative matrices, and use the difference to determine the percentage of estimated bandwidth that is devoted purely to operator loads.  Figure~\ref{fig:percent_bw} shows this estimate for both advection and wave equation kernels at various $N$; as $N$ increases, the operator loads increasingly dominate the usage of available bandwidth.  
%\begin{figure} 
%\centering
%\subfigure[Est. percentage of advection kernel bandwidth]{\includegraphics[width=.45\textwidth]{figs/advec_bw_percent.pdf}}
%\subfigure[Est. percentage of wave kernel bandwidth]{\includegraphics[width=.45\textwidth]{figs/wave_bw_percent.pdf}}
%\caption{Estimated percentage of total bandwidth used for operator loads.}
%\label{fig:percent_bw}
%\end{figure}

This cost may be alleviated by exploiting the tensor-product nature of the orthogonal pyramid basis and quadrature rule.  This was done by Bergot, Cohen, and Durufle in \cite{bergot2013higher} to yield lower-cost evaluations of volume integrals for the non-orthogonal pyramid basis.  
%We propose a partial exploitation of the tensor-product structure, which aims only to reduce the cost of the volume kernel by a factor of $(N+1)$.  The primary cost in the volume kernel is the application of derivative operators to compute $\pd{u}{r}{}$.  However, recalling the mapping from the pyramid coordinates $r,s,t$ to the bi-unit cube coordinates $a,b,c$, 
It may also be possible to decrease memory costs and improve efficiency by adopting quadrature rules constructed directly on the pyramid instead of mapping quadrature rules from the bi-unit cube to the pyramid, which typically involve a fewer number of points than the $(N+1)^3$-point rules currently used \cite{kubatko2013new, witherden2014identification}.  We hope to explore these options in future work.  

\section{Conclusions and acknowledgements}

We have presented a new higher order basis which is orthogonal on vertex-mapped transformations of the reference pyramid, despite the fact that the transformation is non-affine.  This allows for low-storage implementations of discontinuous Galerkin methods on pyramids, which we hope will aid efficient GPU implementations on hex-dominant meshes.   

The work of the first author (JC) was supported partially by the Rice University CAAM Department Pfieffer Postdoctoral Fellowship.  The second author (TW) was supported partially by ANL (award number 1F-32301, subcontract on DOE DE-AC02-06CH11357).  Both authors would like to acknowledge the support of NSF (award number DMS-1216674) in this research.  The authors would additionally like to thank David Medina and Rajesh Gandham for helpful discussions during the writing of this manuscript.  

\bibliography{pyramids}{}
\bibliographystyle{plain}

\end{document}